\UseRawInputEncoding

\documentclass[oneside, 11pt]{amsart} 

\usepackage{amsmath,amsthm,amssymb,epic}
\usepackage{eqlist,eqparbox}
\usepackage{amsfonts}
\usepackage{latexsym}
\usepackage[2emode]{psfrag}
\usepackage{amsthm}
\usepackage{amsmath}
\usepackage[all]{xy}

\newcommand{\Title}[1]{\bigskip\bigskip\centerline{\bf #1}\bigskip}
\newcommand{\Author}[1]{\medskip\centerline{ \it #1}}

\newcommand{\Affiliation}[1]{\medskip\centerline{#1}}
\newcommand{\Email}[1]{\medskip\centerline{#1}\bigskip}

\begin{document}

\newcommand{\N}{\mbox {$\mathbb N $}}
\newcommand{\Z}{\mbox {$\mathbb Z $}}
\newcommand{\Q}{\mbox {$\mathbb Q $}}
\newcommand{\R}{\mbox {$\mathbb R $}}
\newcommand{\lo }{\longrightarrow }
\newcommand{\ul}{\underleftarrow }
\newcommand{\rl}{\underrightarrow }
\newcommand{\rs }{\rightsquigarrow }
\newcommand{\ra }{\rightarrow } 
\newcommand{\dd }{\rightsquigarrow } 
\newcommand{\rars }{\Leftrightarrow }
\newcommand{\ol }{\overline }
\newcommand{\la }{\langle }
\newcommand{\tr }{\triangle }
\newcommand{\xr }{\xrightarrow }
\newcommand{\de }{\delta }
\newcommand{\pa }{\partial }
\newcommand{\LR }{\Longleftrightarrow }
\newcommand{\Ri }{\Rightarrow }
\newcommand{\va }{\varphi }
\newcommand{\Den}{{\rm Den}\,}
\newcommand{\Ker}{{\rm Ker}\,}
\newcommand{\Reg}{{\rm Reg}\,}
\newcommand{\Fix}{{\rm Fix}\,}
\newcommand{\Sup}{{\rm Sup}\,}
\newcommand{\Inf}{{\rm Inf}\,}
\newcommand{\Img}{{\rm Im}\,}
\newcommand{\Id}{{\rm Id}\,}
\newcommand{\ord}{{\rm ord}\,}

\newtheorem{theorem}{Theorem}[section]
\newtheorem{lemma}[theorem]{Lemma}
\newtheorem{proposition}[theorem]{Proposition}
\newtheorem{corollary}[theorem]{Corollary}
\newtheorem{definition}[theorem]{Definition}
\newtheorem{example}[theorem]{Example}
\newtheorem{examples}[theorem]{Examples}
\newtheorem{xca}[theorem]{Exercise}
\theoremstyle{remark}
\newtheorem{remark}[theorem]{Remark}
\newtheorem{remarks}[theorem]{Remarks}
\numberwithin{equation}{section}

\def\leftmark{L.C. Ciungu}


\Title{IMPLICATIVE-ORTHOLATTICES AS ORTHOGONALITY SPACES} 
\title[Implicative-ortholattices as orthogonality spaces]{}
                                                                      
\Author{\textbf{LAVINIA CORINA CIUNGU}}
\Affiliation{Department of Mathematics} 
\Affiliation{St Francis College}
\Affiliation{179 Livingston Street, Brooklyn, New York, NY 11201, USA}
\Email{lciungu@sfc.edu}

\begin{abstract} 
We obtain an orthogonality space by endowing an implicative-ortholattice with a suitable orthogonality relation; for such spaces, we also investigate the particular case of implicative-orthomodular lattices. 
Moreover, we define the commutativity relation between two elements of an implicative-ortholattice, as well as the Sasaki projections on this structure. Furthermore, we characterize the implicative-orthomodular lattices and   implicative-Boolean algebras, showing that the center of an implicative-orthomodular lattice is an implicative-Boolean algebra. We prove that an implicative-ortholattice is an implicative-orthomodular lattice if and only if it admits a full Sasaki set of projections.  
Finally, based on Sasaki maps on implicative-ortholattices, we introduce the notion of Sasaki spaces, proving that 
when a complete implicative-ortholattice admits a full Sasaki set of projections, it is a Sasaki space. 
We also provide a characterization of Dacey spaces arising from implicative-ortholattices. \\

\noindent
\textbf{Keywords:} {implicative-ortholattice, implicative-orthomodular lattice, implicative-Boolean algebra, Sasaki projection, orthogonality space, Sasaki space, Dacey space} \\
\textbf{AMS Subject Classification (2020):} 06C15, 03G25, 06A06, 81P10
\end{abstract}

\maketitle

\section{Introduction} 

A mathematical description of the structure of random events of a quantum mechanical system is one of the most important problem formulated by G. Birkhoff and J. von Neumann in their famous paper ``The Logic of Quantum Mechanics" \cite{Birk1}. 
The behavior of physical systems which describe the events concerning elementary particles is different from that of physical systems observed in classical physics, and another approach had to be developed. 
The first model proposed by G. Birkhoff and J. von Neumann was based on Hilbert spaces and their operators, taking into consideration that in quantum mechanics the state of a physical system is represented by a vector in a Hilbert space. 
The set of all closed subspaces of a Hilbert space forms an orthomodular lattice.  
Therefore, orthomodular posets and orthomodular lattices play a fundamental role in the study of Hilbert spaces, and 
the development of theories on orthomodular structures remains a central topic of research. 
Orthomodular lattices were first introduced by G. Birkhoff and J. von Neumann \cite{Birk1}, and, independently, by 
K. Husimi \cite{Husimi} by studying the structure of the lattice of projection operators on a Hilbert space.  
They introduced the orthomodular lattices as the event structures describing quantum mechanical experiments. 
Later, orthomodular lattices and orthomodular posets were considered as the standard quantum logic \cite{Varad1} (see also \cite{DvPu}). 
A comprehensive overview of quantum structures and quantum logic can be found in handbooks \cite{Egl1, Egl2}. 
Complete studies on orthomodular lattices are presented in the monographs \cite{Beran, Kalm1}. \\
In \cite{Ior30}, \cite{Ior35}, A. Iorgulescu introduced a new framework of algebras of the form $(A,\odot,^*,1)$ 
based on m-BE algebras, which includes ortholattices (redefined as involutive m-BE algebras verifying 
(m-Pimpl) $((x\odot y^*)^*\odot x^*)^*=x$ \cite[Def. 9.2.17]{Ior35}), Boolean algebras (redefined as involutive m-BE algebras verifying (m-Pdiv) $x\odot (x\odot y^*)^*=x\odot y$ \cite[Def. 9.2.27]{Ior35}), quantum MV algebras, 
MV algebras etc.. This is analogous to the existing framework of algebras of the form $(A,\ra,^*,1)$ 
based on BE algebras, which already includes implicative-ortholattices (defined as implicative involutive BE algebras \cite[Def. 3.4.3]{Ior35}), implicative-Boolean algebras (defined in 2009 following the axioms 
of the classical propositional logic and redefined \cite[Def. 3.4.12]{Ior35} as implicative-ortholattices verifying $(@)$ $(y^*\ra x)\ra y=x\ra y$), Wajsberg algebras etc. The two frameworks were connected, in the involutive case, 
by two mutually inverse mappings, $\Phi$ and $\Psi$: $ x\odot y:=(x\ra y^*)^*$ and $x\ra y:=(x\odot y^*)^*$ \cite[Thm. 17.1.1]{Ior35}. Thus, $(impl)\Longleftrightarrow (m$-$Pimpl)$) \cite[Thm. 17.1.1$(4)$]{Ior35}, involutive BE algebras are term-equivalent to involutive m-BE algebras \cite[Cor. 17.1.3]{Ior35}, implicative-ortholattices are term-equivalent to ortholattices \cite[Cor. 17.1.6]{Ior35} and implicative-Boolean algebras are term-equivalent to Boolean algebras \cite[Cor. 17.1.7]{Ior35}. 
Based on implicative-ortholattices, we redefined in \cite{Ciu83} the orthomodular lattices, by introducing 
and studying the implicative-orthomodular lattices (i-OML for short), and we gave certain characterizations of these structures. \\
Implicative-Boolean algebras, implicative-ortholattices and implicative-orthomodular lattices are
implicative versions (i.e. algebras with an implication, $\ra$, of the form $(A,\ra,^*,1)$) of Boolean algebras, ortholattices and orthomodular lattices, respectively. These implicative versions were obtained after redefining the Boolean algebras, ortholattices and orthomodular lattices as algebras of the form $(A,\odot,^*,1)$, in \cite{Ior35}.  
By \cite[Thm. 17.1.1]{Ior35} (the mappings $\Phi$ and $\Psi$), the involutive algebras from the framework of algebras of the form $(A,\ra,^*,1)$ are term-equivalent to the corresponding involutive algebras from the framework of algebras 
of the form $(A,\odot,^*,1)$.
The purpose of introducing these implicative versions is, first of all, to bring the structure closer to the corresponding logic (where the truth is represented by 1 and we have an implication, $\ra$, and a negation, $\neg$) -- which can thus easily derived. \\
J.R. Dacey was the first to study orthogonality spaces from the point of view of the foundations of quantum physics \cite{Dacey} (see also \cite{Paseka2}). Remarkable results on orthogonality spaces were recently presented in \cite{Lind1, Paseka1, Paseka2, Vetter1, Vetter2}. \\
The notion of Sasaki projections on orthocomplemented lattices was introduced by U. Sasaki \cite{Sasaki}, and this 
notion was extended to other structures \cite{Finch, Chajda1}. 
The composition of Sasaki projections on orthomodular lattices was investigated in \cite{Chev}. \\
\indent
In this paper, we obtain an orthogonality space by endowing an implicative-ortholattice with an appropriate orthogonality relation, and we investigate the properties of these spaces. Special results are proved for the particular case of implicative-orthomodular lattices as orthogonality spaces. 
We introduce and study the Sasaki projections, and, based on these notions, we define the commutativity relation $\mathcal{C}$ between two elements of an implicative-ortholattice, proving that an implicative-ortholattice $X$ is an implicative-orthomodular lattice if and only if for all $x,y\in X$, $x\mathcal{C} y$ implies $y\mathcal{C} x$. 
We provide characterizations of the implicative-Boolean algebras, proving that the center of an implicative-orthomodular lattice is an implicative-Boolean algebra. 
We also introduce the notion of Sasaki set of projections, proving that an implicative-ortholattice is an 
implicative-orthomodular lattice if and only if it admits a full Sasaki set of projections. 
Inspired by \cite{Lind1}, we define Sasaki maps on the associated orthogonality space of an 
implicative-ortholattice, and for complete implicative-orthomodular lattices, a method for constructing such 
a map is presented. We also characterize the Dacey spaces defined by implicative-ortholattices.   
Finally, we prove that when a complete implicative-ortholattice admits a full Sasaki set of projections, it is 
a Sasaki space. \\
The paper is organized as follows. 
Section 2 recalls the notions of BE algebras, implicative-ortoholattices, implicative-orthomodular 
lattices and implicative-Boolean algebras used in this paper and presents additional results concerning these 
structures. 
In Section 3, we define an orthogonality relation on an implicative-ortholattice, obtaining an orthogonality space,
and we study the properties of these spaces, proving special results for the particular case of implicative-orthomodular lattices as orthogonality spaces. 
Section 4 is devoted to the study of Sasaki projections and the commutativity between two elements of an 
implicative-ortholattice, and, based on these notions, to providing characterizations of implicative-orthomodular 
lattices. 
In Section 5, we define the implicative-Boolean center of an implicative-ortholattice and, 
using the divisibility relation, we provide characterizations of implicative-Boolean algebras. 
The Sasaki sets of projections are introduced and investigated in Section 6, and it is proved that 
an implicative-ortholattice is an implicative-orthomodular lattice if and only if it admits a full Sasaki 
set of projections. 
Finally, in Section 7, we define and characterize Dacey and Sasaki orthogonality spaces associated with implicative-ortholattices. \\

Studying orthogonality spaces based on implicative-ortholattices may offer a new perspective on quantum 
logic by revealing a more direct connection among various quantum structures that include an implication. 
As benefits of this work, we can mention the following: \\
$\hspace*{0.50cm}$ $-$ Within the framework of orthogonality over implicative-ortholattices, certain results in 
quantum logic may admit simpler and more natural proofs, and new properties may be explored. \\
$\hspace*{0.50cm}$ $-$ When analogues of probabilities are introduced on certain quantum structures, orthogonality plays the role of addition for disjoint sets in the event space. 
In studying states, measures, and state operators on quantum structures with an implication, the new concept of orthogonality seems to be more appropriate to use (see for example \cite{Ciu85}). \\

$\vspace*{1mm}$

\section{Preliminaries}

As we have already mentioned, based on implicative-ortholattices, we redefined the orthomodular lattices 
by introducing the implicative-orthomodular lattices. 
With simpler definitions and signatures, the implicative versions of ortholattices and orthomodular lattices 
allow us to prove new properties and provide new characterizations of these quantum structures. \\
We recall the notions of BE algebras, implicative-ortoholattices, implicative-orthomodular lattices and 
implicative-Boolean algebras that will be used in the paper. 
Additionally, we prove new results regarding implicative-ortholattices and implicative-orthomodular lattices. 

\emph{BE algebras} were introduced in \cite{Kim1} as algebras $(X,\ra,1)$ of type $(2,0)$ satisfying the 
following conditions for all $x,y,z\in X$: 
$(BE_1)$ $x\ra x=1;$ 
$(BE_2)$ $x\ra 1=1;$ 
$(BE_3)$ $1\ra x=x;$ 
$(BE_4)$ $x\ra (y\ra z)=y\ra (x\ra z)$. 
A relation $\le$ is defined by: for all $x, y\in X$, $x\le y$ if and only if $x\ra y=1$. 
A BE algebra $X$ is \emph{bounded} if there exists $0\in X$ such that $0\le x$ for all $x\in X$. 
In a bounded BE algebra $(X,\ra,0,1)$ we define $x^*=x\ra 0$ for all $x\in X$. 
A bounded BE algebra $(X,\ra,0,1)$ with the additional operation ``$^*$" will be denoted by $(X,\ra,^*,1)$ 
(we have $0=1^*)$. 
A bounded BE algebra $X$ is called \emph{involutive} if $x^{**}=x$ for any $x\in X$. 

\begin{lemma} \label{qbe-10} \cite{Ciu83} 
Let $(X,\ra,1)$ be a BE algebra. The following hold for all $x,y,z\in X$: \\
$(1)$ $x\ra (y\ra x)=1;$ 
$(2)$ $x\le (x\ra y)\ra y$. \\
If $X$ is bounded, then: \\
$(3)$ $x\ra y^*=y\ra x^*;$ 
$(4)$ $x\le x^{**}$. \\
If $X$ is involutive, then: \\
$(5)$ $x^*\ra y=y^*\ra x;$ 
$(6)$ $x^*\ra y^*=y\ra x;$ \\
$(7)$ $(x\ra y)^*\ra z=x\ra (y^*\ra z);$ 
$(8)$ $x\ra (y\ra z)=(x\ra y^*)^*\ra z;$ \\ 
$(9)$ $(x^*\ra y)^*\ra (x^*\ra y)=(x^*\ra x)^*\ra (y^*\ra y)$.  
\end{lemma}

\noindent
In a BE algebra $X$, we define the additional operation: \\
$\hspace*{3cm}$ $x\vee_Q y=(x\ra y)\ra y$. \\
If $X$ is involutive, we define the operation: \\
$\hspace*{3cm}$ $x\wedge_Q y=(x^*\vee_Q y^*)^*=((x^*\ra y^*)\ra y^*)^*$, \\
and the relations $\le_Q$ and $\le_L$ by: \\
$\hspace*{3cm}$ for all $x, y\in X$, $x\le_Q y$ iff $x=x\wedge_Q y$ and $x\le_L y$ iff $x=(x\ra y^*)^*$. \\
\noindent
Obviously, $x\le_L y$ iff $x^*=x\ra y^*=y\ra x^*$. \\
Note that the relation $\le_L$ was introduced earlier in \cite[page 213]{Ior35} under the name $\le^P$. 

\begin{proposition} \label{qbe-20} 
Let $X$ be an involutive BE algebra. The following hold for all $x,y,z,u\in X$: \\
$(1)$ $x\le_Q y$ implies $x=y\wedge_Q x$ and $y=x\vee_Q y;$ \\
$(2)$ $\le_Q$ is reflexive and antisymmetric; \\
$(3)$ $x\vee_Q y=(x^*\wedge_Q y^*)^*;$ \\ 
$(4)$ $x\le_Q y$ implies $x\le y;$ \\
$(5)$ $x, y\le_Q z$ and $z\ra x=z\ra y$ imply $x=y$ \emph{(cancellation law)}; \\
$(6)$ $x\le_L y$ implies $x\le y;$ \\ 
$(7)$ $\le_L$ is antisymmetric and transitive; \\
$(8)$ $z\le_L x$ and $z\le_L y$ imply $z\le_L (x\ra y^*)^*;$ \\
$(9)$ $(x\ra y^*)^*\ra (z\ra u^*)=(x\ra z^*)^*\ra (y\ra u^*)$.   
\end{proposition}
\begin{proof}
$(1)$-$(6)$ See \cite{Ciu83, Ciu84}. \\
$(7)$ Since $z\le_L x,y$ implies $z=(z\ra x^*)^*=(z\ra y^*)^*$, we get 
$(z\ra (x\ra y^*))^*=(x\ra (z\ra y^*))^*=(x\ra z^*)^*=(z\ra x^*)^*=z$; hence, $z\le_L (x\ra y^*)^*$. \\
$(8)$ Using Lemma \ref{qbe-10}$(7)$, we have: \\
$\hspace*{2.10cm}$ $(x\ra y^*)^*\ra (z\ra u^*)=x\ra (y\ra (z\ra u^*))$ \\
$\hspace*{6.00cm}$ $=x\ra (z\ra (y\ra u^*))$ \\
$\hspace*{6.00cm}$ $=(x\ra z^*)^*\ra (y\ra u^*)$. 
\end{proof}

An \emph{ortholattice} (OL for short) is an algebra $(X,\wedge,\vee,^*,0,1)$ such that $(X,\wedge,\vee,0,1)$ is a bounded lattice and 
the unary operation $^*$ satisfies the following properties for all $x,y\in X$: 
$(OL_1)$ $(x^*)^*=x$ (double negation); $(OL_2)$ $(x\vee y)^*=x^*\wedge y^*$ and $(x\wedge y)^*=x^*\vee y^*$ 
(De Morgan laws); $(OL_3)$ $x\wedge x^*=0$ (noncontradiction principle); $(OL_4)$ $x\vee x^*=1$ (excluded midle 
principle). 
An \emph{orthomodular lattice} (OML for short) is an ortholattice $(X,\wedge,\vee,^{\prime},0,1)$ verifying: \\
$(OM)$ $(x\wedge y)\vee ((x\wedge y)^{\prime}\wedge x)=x$ for all $x,y\in X$, or, equivalently \\ 
$(OM^{\prime})$ $x\vee (x^{\prime}\wedge y)=y$, whenever $x\le y$ (where $x\le y$ iff $x=x\wedge y$) \cite{PadRud}. \\
The implicative property is old in the literature, it was studied for BCK algebras in 1978 by K. Is$\rm\acute{e}$ki  and S. Tanaka. 

\begin{definition} \label{iol-10} \cite[Def. 2.2.1]{Ior35}
\emph{
A BE algebra is called \emph{implicative} if it satisfies the following condition for all $x,y\in X$: \\
$(impl)$ $(x\ra y)\ra x=x$. 
}
\end{definition}

\begin{definition} \label{iol-10-10} \cite[Def. 3.4.3]{Ior35}
\emph{
An implicative involutive BE algebra is called an \emph{implicative-ortholattice} (i-OL for short). 
}
\end{definition}

\begin{remark} \label{qbe-10-10}
As we mentioned, the relation $\le_L$ was introduced earlier in \cite[page 213]{Ior35} under the name $\le^P$ 
totgether with the binary relations: \\
$\hspace*{1.50cm}$ for all $x, y\in X$, $x\wedge^P y:=(x\ra y^*)^*$ $\hspace*{0.1cm}$ and 
$\hspace*{0.1cm}$ $x\vee^P y:=(x^*\wedge^P y^*)^*=x^*\ra y$. \\
Moreover, in an implicative-ortholattice $(X,\ra,^*,1)$, the binary relation $\le^P$ is a lattice order, the 
bounded lattice being $(X,\wedge^P,\vee^P,0,1)$ \cite[Thm. 3.4.6]{Ior35}. 
\end{remark}

\begin{remark} \label{qbe-10-20}
A. Iorgulescu introduced in \cite{Ior30,Ior35} the notion of involutive m-BE algebras as algebras of the form $(X,\odot,^*,1)$ satisfying the following conditions for all $x,y,z\in X$: 
$(PU)$ $1\odot x=x$; $(Pcomm)$ $x\odot y=y\odot x$; $(Pass)$ $x\odot (y\odot z)=(x\odot y)\odot z$; 
$(m$-$La)$ $x\odot 0=0$; $(m$-$Re)$ $x\odot x^*=0$; $(DN)$ $(x^*)^*=x$, where $0=1^*$. 
The ortholattices were redefined as involutive m-BE algebras verifying $(m$-$Pimpl)$ $((x\odot y^*)^*\odot x^*)^*=x$ \cite[Def. 9.2.17]{Ior35}). 
Using the two mappings $\Phi$ and $\Psi$: $ x\odot y:=(x\ra y^*)^*$ and $x\ra y:=(x\odot y^*)^*$, it follows 
from \cite[Thm. 17.1.1]{Ior35} that the implicative-ortholattices are term-equivalent to the above redefined 
ortholattices, since involutive BE algebras are term-equivalent to involutive m-BE algebras and 
$(impl)\Longleftrightarrow (m$-$Pimpl)$. 
\end{remark}

\begin{remark} \label{qbe-10-30} Note that \cite[Thm. 17.1.1]{Ior35}$(3)$,$(3^{\prime}$,$(3^{\prime\prime})$: \\
$-$ $x\le y$ $\Longleftrightarrow$ $x\le_m y$, where $x\le_m y$ iff $x\odot y^*=0$ \cite[page 295]{Ior35}; \\
$-$ $x\le_Q y$ $\Longleftrightarrow$ $x\le^M_m y$, where $x\le^M_m y$ iff $x\wedge^M_m=x$ 
$\Longleftrightarrow$ $y\odot (y\odot x^*)^*=x$ \cite[page 314]{Ior35}; \\ 
$-$ $x\le_L y$ $\Longleftrightarrow$ $x\le^P_m y$, where $x\le^P_m y$ iff $x\wedge^P_m y=x$ 
               $\Longleftrightarrow$ $x\odot y=x$ \cite[page 309]{Ior35}.                                          
\end{remark}

\begin{lemma} \label{iol-30} \cite{Ciu83} Let $(X,\ra,^*,1)$ be an implicative-ortholattice. 
Then $X$ verifies the following axioms for all $x,y\in X$: \\
$(iG)$       $x^*\ra x=x$, or equivalently, $x\ra x^*=x^*;$ \\
$(Iabs$-$i)$ $(x\ra (x\ra y))\ra x=x;$ \\
$(pi)$       $x\ra (x\ra y)=x\ra y$, \\
where $(iG)\Longleftrightarrow (G)$ $x\odot x=x$.  
\end{lemma}

\begin{remarks} \label{iol-30-05} $(1)$ Let $(X,\ra,^*,1)$ be an involutive BE algebra. Then $\le_L$ is an order relation on $X$ if and only if $X$ satisfies condition $(iG)$. 
Indeed, by Proposition \ref{qbe-20}$(7)$, $X$ is antisymmetric and transitive. 
Moreover, $x\le_L x$ if and only if $x=(x\ra x^*)^*$, that is, $\le_L$ is reflexive if and only if $X$ satisfies $(iG)$. \\
$(2)$ If $X$ is implicative, then $\le_L$ is an order relation on $X$. 
\end{remarks}

\begin{lemma} \label{iol-30-10} \cite{Ciu83}
Let $(X,\ra,^*,1)$ be an involutive BE algebra. The following are equivalent: \\
$(a)$ $X$ is implicative; \\
$(b)$ $X$ verifies axioms $(iG)$ and $(Iabs$-$i)$; \\ 
$(c)$ $X$ verifies axioms $(pi)$ and $(Iabs$-$i)$. 
\end{lemma}

\begin{lemma} \label{iol-30-20} 
Let $X$ be an implicative-ortholattice. The following hold for all $x,y,z,u\in X$: \\ 
$(1)$ $x\le_L y$ iff $y^*\le_L x^*;$ 
$(2)$ $x\le_Q y$ implies $x\le_L y;$ \\
$(3)$ $x\le_L y\ra x$ and $x\le_L x^*\ra y;$ 
$(4)$ $(x\ra y^*)^*\le_L x, y;$ \\
$(5)$ $x^*=x\ra y$ iff $y^*=y\ra x;$ 
$(6)$ $x\wedge_Q y=x$ implies $x\wedge_Q y^*=0;$ \\
$(7)$ $x\le_L y$ implies $x\wedge_Q y^*=0$. \\
If $x\le_L y$, then: \\
$(8)$ $y\ra z\le_L x\ra z$ and $z\ra x\le_L z\ra y;$\\
$(9)$ $x\vee_Q z\le_L y\vee_Q z$ and $x\wedge_Q z\le_L y\wedge_Q z;$ \\
$(10)$ if $x,y\le_L z$, then $x^*\ra y\le_L z;$ \\
$(11)$ $(x\ra y^*)\ra (x\ra y)^*\le_L x;$ \\
$(12)$ if $x\le_L y$ and $z\le_L u$, then $x^*\ra z\le_L y^*\ra u$. 
\end{lemma} 
\begin{proof}
$(1)$ Since $x\le_L y$, we have $x^*=x\ra y^*$; therefore, by $(impl)$, we obtain 
$y^*\ra x=x^*\ra y=(x\ra y^*)\ra y=(y\ra x^*)\ra y=y$, that is, $y^*\le_L x^*$. 
The converse follows similarly. \\
$(2)$ From $x\le_Q y$ we have $x=x\wedge_Q y$, and by $(iG)$ we get  
$x\ra y^*=(x\wedge_Q y)\ra y^*=y\ra (x\wedge_Q y)^*=y\ra ((x^*\ra y^*)\ra y^*)=(x^*\ra y^*)\ra (y\ra y^*)=
(x^*\ra y^*)\ra y^*=(x\wedge_Q y)^*=x^*$. It follows that $x\le_L y$. \\
$(3)$ Applying $(impl)$, we have $x\ra (y\ra x)^*=(y\ra x)\ra x^*=(x^*\ra y^*)\ra x^*=x^*$ and 
$x\ra (x^*\ra y)^*=(x^*\ra y)\ra x^*=x^*$, so $x\le_L y\ra x$ and $x\le_L x^*\ra y$. \\
$(4)$ It follows from $x^*\le_L x\ra y^*$ and $y^*\le_L x\ra y^*$, applying $(1)$. \\
$(5)$ If $x^*=x\ra y$, using $(impl)$ we get: 
$y\ra x=y\ra (x\ra y)^*=(x\ra y)\ra y^*=(y^*\ra x^*)\ra y^*=y^*$. 
Conversely, if $y^*=y\ra x$, then $x\ra y=x\ra (y\ra x)^*=(y\ra x)\ra x^*=(x^*\ra y^*)\ra x^*=x^*$. \\
$(6)$ If $x\wedge_Q y=x$, then $(x^*\ra y^*)\ra y^*=x^*$, so $x^*=y\ra (y\ra x)^*$. Using $(impl)$, we get: 
$x\wedge_Q y^*=((x^*\ra y)\ra y)^*=(((y\ra (y\ra x)^*)\ra y)\ra y)^*=(y\ra y)^*=0$. \\
$(7)$ Since $x\le_L y$ implies $x^*=x\ra y^*$, the proof is similar to $(6)$. \\
$(8)$ Since $x\le_L y$, we have $x^*=x\ra y^*$, so \\
$\hspace*{2.00cm}$ $x\ra z=z^*\ra x^*=z^*\ra (x\ra y^*)=x\ra (z^*\ra y^*)=x\ra (y\ra z)$. \\
Applying $(impl)$, it follows that: \\
$\hspace*{1.00cm}$ $((y\ra z)\ra (x\ra z)^*)^*=((x\ra z)\ra (y\ra z)^*)^*$ \\
$\hspace*{5.00cm}$ $=((x\ra (y\ra z))\ra (y\ra z)^*)^*$ \\
$\hspace*{5.00cm}$ $=(((y\ra z)^*\ra x^*)\ra (y\ra z)^*)^*$ \\
$\hspace*{5.00cm}$ $=(y\ra z)^{**}=y\ra z$, \\
hence, $y\ra z\le_L x\ra z$. 
Moreover, $x\le_L y$ implies $y^*\le_L x^*$, so $x^*\ra z^*\le_L y^*\ra z^*$, that is, $z\ra x\le_L z\ra y$. \\
$(9)$ From $x\le_L y$, using $(8)$, we get $y\ra z\le_L x\ra z$ and $(x\ra z)\ra z\le_L (y\ra z)\ra z$, so   
$x\vee_Q z\le_L y\vee_Q z$. 
Moreover, $y^*\le_L x^*$ implies $y^*\vee_Q z^*\le_L x^*\vee_Q z^*$ hence, $(x^*\vee_Q z^*)^*\le_L (y^*\vee_Q z^*)^*$,  
that is, $x\wedge_Q z\le y\wedge_Q z$. \\
$(10)$ $x,y\le_L z$ implies $z^*\le_L x^*,y^*$, so from $(8)$ and $(iG)$ we have  
$x^*\ra y\le_L z^*\ra y\le_L z^*\ra z=z$. \\
$(11)$ It follows from $(10)$, since $(x\ra y^*)^*\le_L x$ and $(x\ra y)^*\le_L x$. \\
$(12)$ From $x\le_L y$ we have $y^*\le_L x^*$; hence, $x^*\ra z\le_L y^*\ra z\le_L y^*\ra u$.  
\end{proof}

\begin{lemma} \label{ioml-30} Let $X$ be an involutive BE algebra. 
The following are equivalent: \\
$(IOM)$ $x\wedge_Q (y\ra x)=x$ for all $x,y\in X;$ \\
$(IOM^{\prime})$ $x\wedge_Q (x^*\ra y)=x$ for all $x,y\in X;$ $\hspace*{5cm}$ \\
$(IOM^{\prime\prime})$ $x\vee_Q (x\ra y)^*=x$ for all $x,y\in X$. 
\end{lemma}
\begin{proof}
The proof is straightforward. 
\end{proof}

\begin{definition} \label{ioml-30-10} \cite{Ciu83}
\emph{
An \emph{implicative-orthomodular lattice} (i-OML for short) is an implicative-ortholattice satisfying one of 
the equivalent conditions from Lemma \ref{ioml-30}.
}
\end{definition}

\begin{remark} \label{ioml-30-15} 
A. Iorgulescu redefined the orthomodular lattices as ortholattices $(A,\odot,^*,1)$ satisfying 
condition $(Pom)$ $(x\odot y)\oplus ((x\odot y)^*\odot x)=x$, where $x\oplus y=(x^*\odot y^*)^*$. 
In other words, an orthomodular lattice is an involutive m-BE algebra satisfying $(m$-$Pimpl)$ and $(Pom)$ 
\cite[Def. 13.1.4]{Ior35}.   
By \cite[Thm. 17.1.1]{Ior35}, we can prove that $(IOM)\Longleftrightarrow (Pom)$, and, hence, the implicative-orthomodular lattices are term-equivalent to the above redefined orthomodular lattices. 
\end{remark}

\begin{remark} \label{ioml-30-20} 
The implicative-ortholattices $(X,\ra,^*,1)$ are term-equivalent to ortholattices $(X,\wedge,\vee,^{\prime},0,1)$, by the mutually inverse transformations: \\
$\hspace*{3cm}$ $\varphi:$\hspace*{0.2cm}$ x\ra y:=(x\wedge y^{\prime})^{\prime}$ $\hspace*{0.1cm}$ and  
                $\hspace*{0.1cm}$ $\psi:$\hspace*{0.2cm}$ x\wedge y:=(x\ra y^*)^*$, \\
and the relation: for all $x, y\in X$, $x\vee y:=(x^{\prime}\wedge y^{\prime})^{\prime}=x^*\ra y$. 
The unary operation ``$^*$" is defined by: for all $x\in X$, $x^*:=x\ra 0=(x\wedge 0^{\prime})^{\prime}=x^{\prime}$. 
Note that $\le_L$ is the partial order of the lattice $(X,\wedge,\vee,^*,0,1)$. 
Also, by these transformations the implicative-orthomodular lattices are term-equivalent to orthomodular lattices.
\end{remark}

Due to this term-equivalence, the implicative-orthomodular lattices can also be regarded as models for 
quantum logic, and the properties and characterizations of orthomodular lattices can be reformulated and   
established for implicative-orthomodular lattices. 
We review several results on implicative-orthomodular lattices and establish new ones that will be used in the 
next sections. We will see that the connections between relations $\le$, $\le_L$ and $\le_Q$ play a crucial role in 
establishing these results and, consequently, in proving the main results of the paper. 

\begin{theorem} \label{dioml-05-50} \cite{Ciu84} Let $X$ be an implicative-ortholattice. 
The following are equivalent: \\
$(a)$ $X$ is an implicative-orthomodular lattice; \\
$(b)$ $\le_L$ $\subseteq$ $\le_Q;$ \\
$(c)$ $x\le_L y$ implies $y=y\vee_Q x$.  
\end{theorem}

\begin{corollary} \label{dioml-05-55}
If $X$ is an implicative-orthomodular lattice, then $\le_Q = \le_L$. 
\end{corollary}
\begin{proof}
It follows from Theorem \ref{dioml-05-50} and Lemma \ref{iol-30-20}. 
\end{proof}

\begin{proposition} \label{ioml-50} \cite{Ciu83}
Let $X$ be an implicative-orthomodular lattice. Then the following hold for all $x,y,z\in X$: \\
$(1)$ $x\ra (y\wedge_Q x)=x\ra y;$ 
$(2)$ $(x\vee_Q y)\ra (x\ra y)^*=y^*;$ \\ 
$(3)$ $x\wedge_Q ((y\ra x)\wedge_Q (z\ra x))=x;$ 
$(4)$ $(x\ra y)\ra (y\wedge_Q x)=x;$ \\
$(5)$ $x\le_L y$ and $x\le y$ imply $x=y;$ 
$(6)$ $x\wedge_Q y\le_L y\le_L x\vee_Q y;$ \\
$(7)$ $(x\wedge_Q y)\ra (y\wedge_Q x)=1;$ 
$(8)$ $(x\vee_Q y)\ra (y\vee_Q x)=1;$ \\
$(9)$ $(x\vee_Q y)\ra y=x\ra y;$ 
$(10)$ $(x\wedge_Q y)\wedge_Q (y\wedge_Q z)=(x\wedge_Q y)\wedge_Q z$.             
\end{proposition}

\begin{proposition} \label{ioml-50-10} Let $X$ be an implicative-orthomodular lattice. The following hold 
for all $x, y, z\in X$: \\ 
$(1)$ $x\le_L y$, $x\le_L z$ imply $x\le_L y\wedge_Q z;$ \\
$(2)$ $x\le_L y$ implies $(z\wedge_Q y)\wedge_Q x=z\wedge_Q x;$ \\ 
$(3)$ if $x\le y$, $y\le_L x$ imply $x=y;$ \\
$(4)$ $y\le_L x$, $z\le_L x$ imply $y\vee_Q z\le_L x;$ \\ 
$(5)$ $x\ra (x\wedge_Q y)=x\ra y;$ \\
$(6)$ $x\wedge_Q y^*=0$ implies $x\wedge_Q y=x$. 
\end{proposition}
\begin{proof}
$(1)$ From $x\le_L y$, $x\le_L z$ we have $x\le_Q y$, $x\le_Q z$; hence, $x=x\wedge_Q y$ and $x=x\wedge_Q z$. Applying Proposition \ref{ioml-50}$(10)$, we get: $x\wedge_Q (y\wedge_Q z)=(x\wedge_Q y)\wedge_Q (y\wedge_Q z)=(x\wedge_Q y)\wedge_Q z=x\wedge_Q z=x$; hence, $x\le_Q y\wedge_Q z$, that is, $x\le_L y\wedge_Q z$.\\
$(2)$ From $x\le_L y$, we have $x=(x\ra y^*)^*$, and Applying Proposition \ref{ioml-50}$(1)$, it follows that: \\
$\hspace*{1.10cm}$ $(z\wedge_Q y)\wedge_Q x=(((z\wedge_Q y)^*\ra x^*)\ra x^*)^*=(((z\wedge_Q y)^*\ra (x\ra y^*))\ra x^*)^*$ \\
$\hspace*{3.00cm}$ $=((x\ra ((z\wedge_Q y)^*\ra y^*))\ra x^*)^*=((x\ra (y\ra (z\wedge_Q y)))\ra x^*)^*$ \\
$\hspace*{3.00cm}$ $=((x\ra (y\ra z))\ra x^*)^*=((x\ra (z^*\ra y^*))\ra x^*)^*$ \\
$\hspace*{3.00cm}$ $=((z^*\ra (x\ra y^*))\ra x^*)^*=((z^*\ra x^*)\ra x^*)^*=z\wedge_Q x$. \\
$(3)$ From $x\le y$, we get $y\wedge_Q x=((y^*\ra x^*)\ra x^*)^*=((x\ra y)\ra x^*)^*=(1\ra x^*)^*=x$. 
On the other hand, since $y\le_L x$, we have $y^*=y\ra x^*$, and we get: \\
$\hspace*{2.00cm}$ $y\wedge_Q x=((y^*\ra x^*)\ra x^*)^*=(((y\ra x^*)\ra x^*)\ra x^*)^*$ \\
$\hspace*{3.00cm}$ $=((y\vee_Q x^*)\ra x^*)^*=(x\ra (y^*\wedge_Q x))^*$ \\
$\hspace*{3.00cm}$ $=(x\ra y^*)^*=y$ (by Prop. \ref{ioml-50}$(1)$). \\
It follows that $x=y$. \\  
$(4)$ From $y, z\le_L x$, we have $x^*\le_L y^*, z^*$, and applying $(1)$, we get $x^*\le_L y^*\wedge_Q z^*=(y\vee_Q z)^*$, 
that is, $y\vee_Q z\le_L x$. \\ 
$(5)$ Since by $(impl)$, $x=((x^*\ra y)\ra x^*)^*=(x\ra (x^*\ra y)^*)^*$, we get: \\
$\hspace*{2.00cm}$ $x\ra (x\wedge_Q y)=x\ra ((x^*\ra y^*)\ra y^*)^*$ \\ 
$\hspace*{4.10cm}$ $=(x\ra (x^*\ra y)^*)^*\ra ((x^*\ra y^*)\ra y^*)^*$ \\ 
$\hspace*{4.10cm}$ $=x\ra ((x^*\ra y^*)\ra ((x^*\ra y^*)\ra y^*)^*)$ (Lemma \ref{qbe-10}$(7)$) \\
$\hspace*{4.10cm}$ $=x\ra (((x^*\ra y^*)\ra y^*)\ra (x^*\ra y^*)^*)$ (Lemma \ref{qbe-10}$(6)$) \\
$\hspace*{4.10cm}$ $=x\ra ((y\ra (x^*\ra y^*)^*)\ra (x^*\ra y^*)^*)$ (Lemma \ref{qbe-10}$(6)$) \\
$\hspace*{4.10cm}$ $=x\ra (y\vee_Q (x^*\ra y^*)^*)$ \\
$\hspace*{4.10cm}$ $=x\ra (y^*\wedge_Q (x^*\ra y^*))^*$ \\
$\hspace*{4.10cm}$ $=x\ra (y^*)^*$ (by $(IOM)$) \\
$\hspace*{4.10cm}$ $=x\ra y$. \\
$(6)$ Assume that $x\wedge_Q y^*=0$; then $(x^*\ra y)\ra y=1$, and therefore $x^*\ra y\le y$. 
On the other hand, $y\le_L x^*\ra y$, and by $(3)$, we get $y=x^*\ra y=y^*\ra x$. 
It follows that: \\
$\hspace*{2.00cm}$ $x\wedge_Q y=((x^*\ra y^*)\ra y^*)^*=((y\ra x)\ra y^*)^*=((x^*\ra y)\ra x)\ra y^*)^*$ \\
$\hspace*{3.00cm}$ $=(((y^*\ra x)\ra x)\ra y^*)^*=((y^*\vee_Q x)\ra y^*)^*$ \\
$\hspace*{3.00cm}$ $=(y\ra (y\wedge_Q x^*))^*=(y\ra x^*)^*$ (by $(5)$). \\
$\hspace*{3.00cm}$ $=((x^*\ra y)\ra x^*)^*=(x^*)^*=x$ (by $(impl)$). 
\end{proof}

Characterizations of implicative-orthomodular lattices are important in the study of quantum logic, as they support the computational analysis of logic and help building proof systems and formal calculi for quantum logic. 
Implication based characterizations often make easier to connect the semantics of quantum logic with 
algorithmic or syntactic frameworks. 
The next theorem offers new characterizations of implicative-orthomodular lattices and plays an important role in proving many of the results presented in this paper. Interestingly, an i-OL is an i-OML precisely when $x\le y$ and $y\le_L x$ entail $x=y$. 

\begin{theorem} \label{dioml-05-60} Let $X$ be an implicative-ortholattice. 
The following are equivalent: \\
$(a)$ $X$ is an implicative-orthomodular lattice; \\ 
$(b)$ $(x\ra y)\ra (y\wedge_Q x)=x$ for all $x,y\in X;$ \\
$(c)$ $x\le y$ and $y\le_L x$ imply $x=y;$ \\
$(d)$ $x\wedge_Q y^*=0$ implies $x\wedge_Q y=x;$ \\
$(e)$ $x\ra (x\wedge_Q y)=x\ra y$ for all $x,y\in X$.   
\end{theorem}
\begin{proof}
$(a)\Leftrightarrow (b)$ If $X$ is an i-OML and $x,y\in X$, by Proposition \ref{ioml-50}$(4)$, 
$(x\ra y)\ra (y\wedge_Q x)=x$. 
Conversely, let $X$ be an i-OL such that $(x\ra y)\ra (y\wedge_Q x)=x$ for all $x,y\in X$. 
Then we have: \\
$\hspace*{2.00cm}$ $x\vee_Q (x\ra y)^*=(x\ra (x\ra y)^*)\ra (x\ra y)^*$ \\ 
$\hspace*{4.20cm}$ $=((x\ra y)\ra x^*)\ra (x\ra y)^*$ \\
$\hspace*{4.20cm}$ $=((y^*\ra x^*)\ra x^*)\ra (x\ra y)^*$ \\
$\hspace*{4.20cm}$ $=(y\wedge_Q x)^*\ra (x\ra y)^*$ \\
$\hspace*{4.20cm}$ $=(x\ra y)\ra (y\wedge_Q x)=x$. \\
Hence, $X$ satisfies condition $(IOM^{''})$, that is, $X$ is an i-OML. \\
$(a)\Leftrightarrow (c)$ Let $X$ be an i-OML and let $x,y\in X$ such that $x\le y$ and $y\le_L x$. 
Then by Proposition \ref{ioml-50-10}$(3)$, $x=y$.  
Conversely, let $X$ be an i-OL satisfying condition $(c)$ and let $x,y\in X$ such that 
$x\le_L y$. We have 
$(x\wedge_Q y)\ra x=x^*\ra (x\wedge_Q y)^*=x^*\ra ((x^*\ra y^*)\ra y^*)=(x^*\ra y^*)\ra (x^*\ra y^*)=1$; 
hence, $x\wedge_Q y\le x$. 
Since $x\le_L y$, $x=(x\ra y^*)^*$, so $x^*=x\ra y^*$. It follows that: \\
$\hspace*{2.10cm}$ $(x\ra (x\wedge_Q y)^*)^*=(x\ra ((x^*\ra y^*)\ra y^*))^*$ \\
$\hspace*{4.80cm}$ $=((x^*\ra y^*)\ra (x\ra y^*))^*$ \\
$\hspace*{4.80cm}$ $=((x^*\ra y^*)\ra x^*)^*=(x^*)^*=x$ (by $(impl)$), \\
thus, $x\le_L x\wedge_Q y$.  By hypothesis, $x\wedge_Q y\le x$ and $x\le_L x\wedge_Q y$ implies $x=x\wedge_Q y$, that is, $x\le_Q y$. 
We proved that $x\le_L y$ implies $x\le_Q y$, hence, by Theorem \ref{dioml-05-50}, it follows that $X$ is 
an i-OML. \\
$(a)\Leftrightarrow (d)$ Let $X$ be an i-OML and let $x,y\in X$ such that $x\wedge_Q y^*=0$.  
Then by Proposition \ref{ioml-50-10}$(6)$, $x\wedge_Q y^*=0$ implies $x\wedge_Q y=x$.  
Conversely, let $X$ be an i-OL satisfying condition $(d)$. 
For all $x,y\in X$, using $(iG)$ we have: \\
$\hspace*{2.00cm}$ $x\wedge_Q (y\ra x)^*=((x^*\ra (y\ra x))\ra (y\ra x))^*$ \\
$\hspace*{4.20cm}$ $=((y\ra (x^*\ra x))\ra (y\ra x))^*$ \\
$\hspace*{4.20cm}$ $=((y\ra x)\ra (y\ra x))^*=0$. \\
By hypothesis, $x\wedge_Q (y\ra x)=x$, hence, $X$ satisfies condition $(IOM)$, and thus $X$ is an i-OML. \\
$(a)\Leftrightarrow (e)$ If $X$ is an i-OML, by Proposition \ref{ioml-50-10}$(5)$ condition $(e)$ is satisfied. 
Conversely, assume that $X$ is an i-OL verifying condition $(e)$. 
It follows that $x\ra (x\wedge_Q (y\ra x^*))=x\ra (y\ra x^*)$, and using $(impl)$ we get 
$x\ra (x\ra (y\ra x^*)^*)^*=y\ra x^*$ for all $x,y\in X$. Replacing $x$ by $x\ra y$ and $y$ by $x^*$ and 
applying again $(impl)$, we get $((x\ra y)\ra ((x\ra y)\ra x^*)^*=x$. 
Finally, since $x\vee_Q (x\ra y)^*)=x$ for all $x,y\in X$, it follows that $X$ satisfies condition $(IOM^{\prime\prime})$. Hence, $X$ is an i-OML. 
\end{proof}

\begin{remarks} \label{dioml-05-70} Let $X$ be an implicative-orthomodular lattice and let $x,y\in X$. \\
$(1)$ By Lemma \ref{iol-30-20}$(6)$ and Proposition \ref{ioml-50-10}$(6)$, $x\wedge_Q y=x$ iff $x\wedge_Q y^*=0$, 
or equivalently, $x\le_Q y$ iff $x\wedge_Q y^*=0$. \\
$(2)$ By $(1)$ and Corollary \ref{dioml-05-55}, $x\le_Q y$ iff $x\le_L y$ iff $x\wedge_Q y^*=0$. 
\end{remarks} 

A typical example of non-orthomodular lattice is the ortholattice $\mathbf{O}_6=(O_6,\wedge,\vee,^*,0,1)$, where  
$O_6=\{0,x,y,x^*,y^*,1\}$ (Benzene ring; see for example \cite{Bonzio1}). 
We adapt this example for the case of implicative-ortholattices.  

\begin{example} \label{dioml-60}   
Consider the set $X=\{0,a,b,c,d,1\}$ and the operations $\ra$, $\wedge_Q$ defined below. 
\[
\begin{picture}(50,-70)(0,60)
\put(37,11){\circle*{3}}
\put(34,0){$0$}
\put(37,11){\line(3,4){20}}
\put(57,37){\circle*{3}}
\put(61,35){$d$}
\put(37,11){\line(-3,4){20}}
\put(18,37){\circle*{3}}
\put(8,35){$a$}
\put(18,37){\line(0,1){30}}
\put(18,68){\circle*{3}}
\put(8,68){$b$}
\put(57,37){\line(0,1){30}}
\put(57,68){\circle*{3}}
\put(61,68){$c$}
\put(18,68){\line(3,4){20}}
\put(38,95){\circle*{3}}
\put(35,100){$1$}
\put(57,68){\line(-3,4){20}}
\end{picture}
\hspace*{2cm}
\begin{array}{c|cccccc}
\rightarrow & 0 & a & b & c & d & 1 \\ \hline
0 & 1 & 1 & 1 & 1 & 1 & 1 \\
a & c & 1 & 1 & c & c & 1 \\
b & d & 1 & 1 & c & d & 1 \\
c & a & a & b & 1 & 1 & 1 \\
d & b & b & b & 1 & 1 & 1 \\
1 & 0 & a & b & c & d & 1
\end{array}
\hspace{10mm}
\begin{array}{c|ccccccc}
\wedge_Q & 0 & a & b & c & d & 1 \\ \hline
0    & 0 & 0 & 0 & 0 & 0 & 0 \\ 
a    & 0 & a & b & 0 & 0 & a \\ 
b    & 0 & a & b & 0 & 0 & b \\ 
c    & 0 & 0 & 0 & c & d & c \\
d    & 0 & 0 & 0 & c & d & d \\
1    & 0 & a & b & c & d & 1 
\end{array}
.
\]

$\vspace*{5mm}$

Then $(X,\ra,^*,1)$, where $x^*=x\ra 0$, is an i-OL, but not an i-OML. Indeed, $a\wedge_Q (d\ra a)=b\neq a$; hence, condition $(IOM)$ is not satisfied. We have the following remarks: \\
$(1)$ Since $(a\ra b^*)^*=a$ and $a\wedge_Q b=b\neq a$, we have $a\le_L b$ and $a\nleq_Q b$. 
Moreover, $a\le_L b$, but $b\vee_Q a=(b^*\wedge_Q a^*)^*=(d\wedge_Q c)^*=c^*=a\neq b$. 
These observations confirm that conditions $(b)$ and $(c)$ of Theorem \ref{dioml-05-50} are satisfied for 
all $x, y\in X$ only if $X$ is an i-OML.\\
$(2)$ We also have: \\
$\hspace*{3cm}$ $(b\ra c)\ra (c\wedge_Q b)=a\neq b;$ \\
$\hspace*{3cm}$ $b\le a$ and $a\le_L b$, but $a\neq b;$ \\
$\hspace*{3cm}$ $a\wedge_Q b^*=0$ and $a\wedge_Q b=b\neq a;$ \\
$\hspace*{3cm}$ $b\ra (b\wedge_Q c)=d\neq c=b\ra c$. \\
Therefore, conditions $(b)$, $(c)$, $(d)$ of Theorem \ref{dioml-05-60} are satisfied for all $x, y\in X$ 
only if $X$ is an i-OML.
\end{example}

\noindent
The following example is derived from \cite[Ex. 13.3.2]{Ior35}. 

\begin{example} \label{dioml-60-10}
Let $X=\{0,a,b,c,d,e,f,g,h,1\}$ and let $(X,\ra,^*,1)$ be the i-OL with $\ra$ and the 
corresponding operation $\wedge_Q$ given in the following tables: \\
\[
\begin{array}{c|cccccccccc}
\ra & 0 & a & b & c & d & e & f & g & h & 1 \\ \hline
0   & 1 & 1 & 1 & 1 & 1 & 1 & 1 & 1 & 1 & 1 \\ 
a   & b & 1 & b & 1 & 1 & h & f & f & h & 1 \\ 
b   & a & a & 1 & 1 & 1 & a & 1 & a & 1 & 1 \\ 
c   & d & 1 & 1 & 1 & d & 1 & 1 & 1 & 1 & 1 \\
d   & c & 1 & 1 & c & 1 & 1 & 1 & 1 & 1 & 1 \\
e   & f & 1 & f & 1 & 1 & 1 & f & f & 1 & 1 \\
f   & e & a & h & 1 & 1 & e & 1 & a & h & 1 \\
g   & h & 1 & h & 1 & 1 & h & 1 & 1 & h & 1 \\
h   & g & a & f & 1 & 1 & a & f & g & 1 & 1 \\
1   & 0 & a & b & c & d & e & f & g & h & 1
\end{array}
\hspace{10mm}
\begin{array}{c|cccccccccc}
\wedge_Q & 0 & a & b & c & d & e & f & g & h & 1 \\ \hline
0    & 0 & 0 & 0 & 0 & 0 & 0 & 0 & 0 & 0 & 0 \\ 
a    & 0 & a & 0 & c & d & e & g & g & e & a \\ 
b    & 0 & 0 & b & c & d & 0 & b & 0 & b & b \\ 
c    & 0 & a & b & c & 0 & e & f & g & h & c \\
d    & 0 & a & b & 0 & d & e & f & g & h & d \\
e    & 0 & e & 0 & c & d & e & 0 & 0 & e & e \\
f    & 0 & g & b & c & d & 0 & f & g & b & f \\
g    & 0 & g & 0 & c & d & 0 & g & g & 0 & g \\
h    & 0 & e & b & c & d & e & b & 0 & h & h \\
1    & 0 & a & b & c & d & e & f & g & h & 1
\end{array}
.
\]

Then $(X,\ra,^*,1)$ is an i-OML, and this example can be used to verify the results presented in this section. 
\end{example}

A major distinction between classical logic, modeled by Boolean algebras, and quantum logic, modeled by 
orthomodular lattices, is that distributivity fails in quantum logic. 
D.J. Foulis \cite{Foulis2} and S.S. Holland \cite{Holl1} independently established 
conditions under which a triple of elements in an orthomodular lattice is distributive. 
We defined distributivity in an implicative-orthomodlular lattice in \cite{Ciu84} and proved an analogue of the
Foulis-Holland theorem for such structures. 

\begin{definition} \label{dioml-70} \cite{Ciu84}
\emph{
An implicative-ortholattice $X$ is \emph{distributive} if it satisfies one of the following equivalent conditions: 
for all $x, y, z\in X$, \\
$(Idis_1)$ $((x^*\ra y)\ra z^*)^*=(x\ra z^*)\ra (y\ra z^*)^*$, \\
$(Idis_2)$ $((x\ra y^*)\ra z)^*=((z^*\ra x)\ra (z^*\ra y)^*$. 
}
\end{definition}

Using the mutually inverse transformations $\varphi$ and $\psi$ from 
Remark \ref{ioml-30-20}, we can easily see that $(Idis_1)$ and $(Idis_2)$ are equivalent to the lattice distributivity 
laws: for all $x, y, z\in X$, \\
$(Dis_1)$ $z\wedge (x\vee y)=(z\wedge x)\vee (z\wedge y)$ and 
$(Dis_2)$ $z\vee (x\wedge y)=(z\vee x)\wedge (z\vee y)$, respectively. 

\begin{proposition} \label{dioml-90} \cite{Ciu84} 
\emph{
An implicative-orthomodular lattice $X$ is distributive if and only if it satisfies the 
\emph{divisibility condition}: for all $x, y\in X$, \\
$(Idiv)$ $x\ra (x\ra y)^*=x\ra y^*$.  
}
\end{proposition}

A \emph{Boolean algebra} is an ortholattice $(X,\wedge,\vee,^*,0,1)$ satisfying one of the equivalent conditions $(Dis_1)$ and $(Dis_2)$. 
Motivated by the axioms system of the classical propositional logic, A. Iorgulescu introduced the implicative-Boolean algebras in 2009 and proved that they are term-equivalent to Boolean algebras. Consequently, implicative-Boolean algebras may also be regarded as models of classical logic. 
The following definition is an equivalent of \cite[Def. 3.3.23]{Ior35} for implicative-Boolean algebras. 

\begin{definition} \label{dioml-100} \cite[Def. 3.4.12]{Ior35}
\emph{
An implicative-ortholattice satisfying condition: for all $x, y\in X$, 
$(@)$ $(y^*\ra x)\ra y=x\ra y$, 
is called an \emph{implicative-Boolean algebra} (\emph{i-Boolean algebra} for short). 
}
\end{definition}

In other words, i-Boolean algebra = i-OL + $(@)$. \\
Using the mutually inverse transformations $\varphi$ and $\psi$ from Remark \ref{ioml-30-20}, the implicative-Boolean algebras are term-equivalent to Boolean algebras \cite[Prop. 3.3.24]{Ior35}. \\
According to \cite[Thm. 9.2.12]{Ior35}, in any ortholattice, condition $(Dis_1)$ is equivalent to condition: 
for all $x, y\in X$, \\ 
$(Div)$ $x\wedge (x\wedge y^*)^*=x\wedge y$. \\
According to \cite[Thm. 9.2.25]{Ior35}, the class of ortholattices satisfying one of the equivalent conditions 
$(Div)$, $(Dis_1)$ or $(Dis_2)$ is equivalent to the class of Boolean algebras. \\ 
We also mention that, by \cite[Thm. 9.2.26]{Ior35}, a Boolean algebra is an involutive m-BE algebra 
$(X, \odot, ^*, 1)$ satisfying condition $(m$-$Pdiv)$ $x\odot (x\odot y^*)^*=x\odot y$. 

\begin{remark} \label{cioml-80} Let $X$ be an implicative-ortholattice and let $x,y\in X$. 
Conditions $(Idiv)$ and $(@)$ are equivalent. 
Indeed, by $(Idiv)$, $x\ra y=y^*\ra x^*=y^*\ra (y^*\ra x)^*=(y^*\ra x)\ra y$, that is, $(@)$. 
Conversely, $(@)$ implies $x\ra y^*=y\ra x^*=(x\ra y)\ra x^*=x\ra (x\ra y)^*$, that is, $(Idiv)$. 
\end{remark}

Consequently, i-Boolean algebra = i-OL + $(Idiv)$.  
Taking the above results into account, we obtain an analogue of \cite[Thm. 9.2.25]{Ior35}: 

\begin{theorem} \label{dioml-110} The class of implicative-ortholattices satisfying one of the equivalent conditions $(Idiv)$, $(Idis_1)$ or $(Idis_2)$ is term-equivalent to the class of implicative-Boolean algebras. 
\end{theorem}

$\vspace*{1.0mm}$

\section{Implicative-ortholattices as orthogonality spaces}

Most algebraic structures that serve as foundations of physics, such as projective Hilbert spaces and orthomodular 
lattices, can be endowed with orthogonality relations giving rise to orthogonality spaces. 
Orthogonality spaces are not themselves models for quantum logic, but they determine other features of such models.  
Quantum structures can often be uniquely obtained from the orthogonality relations defined on them (as is the case, 
for example, with ortholattices). 
The characterization of quantum structures solely based on orthogonality relations was proposed by D.J. Foulis 
in 1960s, while J.R. Dacey studied the orthogonality spaces from the point of view of quantum physics \cite{Dacey}. 
Since then, orthogonality spaces have been investigated by many authors, but the subject is still of 
current interest. 
This motivated us to redefine certain orthomodular structures and their corresponding orthogonality relations, 
with the intention of creating a new framework to study the connection between these structures and quantum logic. \\ 
Given an implicative-ortholattice $(X,\ra,^*,1)$, we define an orthogonality relation 
$\perp$ on $X^{\prime}=X\setminus \{0\}$, and we prove that $(X^{\prime},\perp)$ is an orthogonality space. 

\begin{definition} \label{ooml-05} \cite{Lind1}
\emph{
An \emph{orthogonality space} (or an \emph{orthoset}) $(X,\perp)$ is a set $X$ equipped with a symmetric and irreflexive binary relation $\perp$, called the \emph{orthogonality relation}. 
We call $x,y\in X$ \emph{orthogonal elements} if $x\perp y$. 
}
\end{definition}

If $(X,\perp)$ is an orthogonality space, and $A\subseteq X$, denote $A^{\perp}=\{x\in X\mid x\perp a$ 
for all $a\in A\}$. 
For $A=\{x\}$ ($A$ is a singleton), we denote $x^{\perp}$ instead of $\{x\}^{\perp}$. 
The subset $A$ is called \emph{orthoclosed} if $A=A^{\perp\perp}$, and denote by $\mathcal{CL}(X,\perp)$ the set of 
all orthoclosed subsets of $X$. Obviously, $X^{\perp}=\emptyset$ and $\emptyset^{\perp}=X$. 
A subset of mutually orthogonal elements of an orthogonality space is called a \emph{$\perp$-set}. 

\begin{definition} \label{ooml-10}
\emph{
Given an implicative-ortholattice $(X,\ra,^*,1)$, we define the binary relation $\perp$ by: for all $x, y\in X$,  
$x\perp y$ if and only if $x^*=x\ra y$.  
}
\end{definition}

\begin{lemma} \label{ooml-20} Let $X$ be an implicative-ortholattice. 
The following hold for all $x,y\in X$: \\
$(1)$ $x\perp y$ iff $y\perp x;$ $(2)$ $x\perp x$ iff $x=0;$ $(3)$ $0\perp x;$ $(4)$ $1\perp x$ iff $x=0;$ \\
$(5)$ $x\le_L y$ implies $x\perp y^*;$ $(6)$ $x\perp (y\ra x)^*;$ $(7)$ $x\perp y$ iff $x\le_L y^*$. 
\end{lemma}
\begin{proof}
$(1)$ follows from Lemma \ref{iol-30-20}$(5)$, and $(2)$-$(5)$ are straightforward. \\
$(6)$ By $(impl)$, $x\ra (y\ra x)^*=(y\ra x)\ra x^*=(x^*\ra y^*)\ra x^*$; hence, $x\perp (y\ra x)^*$. \\
$(7)$ We have $x\perp y$ iff $x^*=x\ra y$ iff $x=(x\ra y)^*$ iff $x\le_L y^*$. 
\end{proof}

\begin{lemma} \label{ooml-30} Let $X$ be an implicative-ortholattice and let $x, y\in X$ such that $x\perp y$.  
The following hold: \\
$(1)$ $x^*\ra y^*=y^*$ and $y^*\ra x^*=x^*;$ \\ 
$(2)$ $(x^*\ra y)\ra x=y^*;$ \\
$(3)$ $(x^*\ra y)\ra x=x^*;$ \\
$(4)$ $x^*\ra (x^*\ra y)^*=y^*$. 
\end{lemma}
\begin{proof}
Let $x, y\in X$ such that $x\perp y$, that is $y\perp x$. Hence, $x^*=x\ra y$ and $y^*=y\ra x$. 
According to Proposition \ref{ioml-50}$(9)$, $(x\vee_Q y)\ra y=x\ra y$ and we have: \\
$(1)$ $x^*\ra y^*=y\ra x=y^*$ and $y^*\ra x^*=x\ra y=y^*$. \\
$(2)$ $(x^*\ra y)\ra x=(y^*\ra x)\ra x=((y\ra x)\ra x)\ra x=(y\vee_Q x)\ra x=y\ra x=y^*$. \\
$(3)$ $(x^*\ra y)\ra y=((x\ra y)\ra y)\ra y=(x\vee_Q y)\ra y=x\ra y=y^*$. \\
$(4)$ $x^*\ra (x^*\ra y)^*=(x^*\ra y)\ra x=(y^*\ra x)\ra x=((y\ra x)\ra x)\ra x=(y\vee_Q x)\ra x=y\ra x=y^*$.  
\end{proof}

\begin{proposition} \label{ooml-50} Let $X$ be an implicative-orthomodular lattice and let $x,y\in X$. 
Then $x\perp y$ if and only if $x\wedge_Q y=0$.   
\end{proposition}
\begin{proof}
Assume that $x\perp y$; hence, $x^*=x\ra y$, and by $(impl)$, we obtain: \\ 
$\hspace*{2.00cm}$  $x\wedge_Q y=((x^*\ra y^*)\ra y^*)^*=(((x\ra y)\ra y^*)\ra y^*)^*$ \\
$\hspace*{3.0cm}$  $=(((y^*\ra x^*)\ra y^*)\ra y^*)^*=(y^*\ra y^*)^*=0$. \\
Conversely, if $x\wedge_Q y=0$, then $(x^*\ra y^*)\ra y^*=1$; hence, $x^*\ra y^*\le y^*$. 
On the other hand, $y^*\le_L x^*\ra y^*$, and by Proposition \ref{ioml-50-10}$(3)$, we get $y^*=x^*\ra y^*=y\ra x$. 
Thus, $y\perp x$, and hence $x\perp y$. 
\end{proof}

In the next result, we give a characterization of implicative-orthomodular lattices using the orthogonality relation. 

\begin{proposition} \label{ooml-40} Let $X$ be an implicative-ortholattice. Then $X$ is an implicative-orthomodular 
lattice if and only if for all $x, y\in X$, $x\perp y$ implies $x\wedge_Q y^*=x$. 
\end{proposition}
\begin{proof}
Let $X$ be an i-OML and let $x,y\in X$ such that $x\perp y$.  
By Proposition \ref{ooml-50}, $x\wedge_Q y=0$, and applying Theorem \ref{dioml-05-60} we get $x\wedge_Q y^*=x$. 
Conversely, suppose that for any $x,y\in X$ such that $x\perp y$, we have $x\wedge_Q y^*=x$. 
Since by Lemma \ref{ooml-20}$(6)$ we have $x\perp (y\ra x)^*$, it follows that $x\wedge_Q (y\ra x)=x$
Hence, $X$ satisfies condition $(IOM)$, and therefore it is an i-OML. 
\end{proof}

Let $(X,\ra,^*,1)$ be an implicative-ortholattice and let $X^{\prime}=X\setminus \{0\}$. 
According to Lemma \ref{ooml-20}, $\perp$ is an orthogonality relation on $X^{\prime}=X\setminus \{0\}$; hence,  $(X^{\prime},\perp)$ is an orthogonality space, called the \emph{associated orthogonality space} to the 
implicative-ortholattice. \\
If $A,B\subseteq X^{\prime}$, denote: $A\wedge B=A\cap B$, $A\vee B=(A^{\perp}\wedge B^{\perp})^{\perp}$. 
Clearly, $A\subseteq B$ implies $B^{\perp}\subseteq A^{\perp}$. \\
For any $A,B\in \mathcal{CL}(X^{\prime},\perp)$, define $A\le_L B$ iff $A\subseteq B$ and 
$A\ra B=(A\cap B^{\perp})^{\perp}$. 
We can see that $A^*=A\ra \emptyset=(A\cap \emptyset^{\perp})^{\perp}=A^{\perp}$. 

\begin{remarks} \label{dacey-10} $(1)$ For any $A\subseteq X^{\prime}$, we have $A\subseteq A^{\perp\perp}$ and 
$A^{\perp\perp\perp}=A^{\perp}$. \\
Indeed, let $x\in A$. For any $y\in A^{\perp}$, we have $y\perp A$, that is, $x\perp y$. 
Hence, $x\in A^{\perp\perp}$, and $A\subseteq A^{\perp\perp}$. 
It follows that $A^{\perp\perp\perp}\subseteq A^{\perp}$, and $A\subseteq A^{\perp\perp}$ implies also 
$A^{\perp}\subseteq (A^{\perp})^{\perp\perp}=A^{\perp\perp\perp}$. 
Hence, $A^{\perp\perp\perp}=A^{\perp}$. \\
$(2)$ For any $A,B\in \mathcal{CL}(X^{\prime},\perp)$, the definition of $A\le_L B$ is justified by the fact that  
$A\subseteq B$ iff $A=A\cap B=(A\ra B^{\perp})^{\perp}=(A\ra B^*)^*$ iff $A\le_L B$.  
\end{remarks}

\begin{remark} \label{dacey-10-10} For any $A,B \in \mathcal{CL}(X^{\prime},\perp)$, we have 
$A^{\perp}, A\cap B, A\vee B\in \mathcal{CL}(X^{\prime},\perp)$. \\ 
Indeed, $A^{\perp}=A^{\perp\perp\perp}$ implies $A^{\perp}\in \mathcal{CL}(X^{\prime},\perp)$. 
Since $A\cap B\subseteq A,B$ implies $(A\cap B)^{\perp\perp}\subseteq A,B$, we get  
$(A\cap B)^{\perp\perp}\subseteq A\wedge_Q B$. On the other hand, $A\cap B\subseteq (A\cap B)^{\perp\perp}$; therefore  
$(A\cap B)^{\perp\perp}= A\wedge_Q B$ and $A\cap B \in \mathcal{CL}(X^{\prime},\perp)$. 
Since $A\vee B=(A^{\perp}\cap B^{\perp})^\perp$, it follows that $A\vee B\in \mathcal{CL}(X^{\prime},\perp)$.
\end{remark}

\begin{proposition} \label{dacey-20} Let $(X,\ra,^*,1)$ be an implicative-ortholattice and let 
$(X^{\prime},\perp)$ be the associated orthogonality space.
Then $(\mathcal{CL}(X^{\prime},\perp),\ra,^*,X^{\prime})$ is an implicative-ortholattice. 
\end{proposition}
\begin{proof}
Obviously, $\mathcal{CL}(X^{\prime},\perp)$ is involutive and it verifies the axioms of BE algebras. 
Moreover, $(A\ra B)\ra A=((A\cap B^{\perp})^{\perp}\cap A^{\perp})^{\perp}=(A^{\perp})^{\perp}=A$, 
since $A\cap B^{\perp}\subseteq A$ implies $A^{\perp}\subseteq (A\cap B^{\perp})^{\perp}$.  
Hence, axiom $(impl)$ is also verified.               
\end{proof}

As shown in Proposition \ref{dacey-20}, starting from the set of all orthoclosed subsets of an orthogonality space,  one can construct a quantum structure (in our case an implicative-ortholattice). We will see that, under additional conditions such as $(IOM)$, $(IOM^{\prime})$ or $(IOM^{\prime\prime})$, this structure becomes an implicative-orthomodular lattice, that is, a model of quantum logic.  
Consequently, orthogonality spaces can characterize quantum logic abstractly, without assuming an underlying Hilbert space. 

\begin{example} \label{ooml-60}
Consider the i-OL $(X,\ra,^*,1)$ from Example \ref{dioml-60}, where $X=\{0,a,b,c,d,1\}$. 
Denoting $X^{\prime}=X\setminus \{0\}$, then $(X^{\prime},\perp)$ is an orthogonality space. \\
$(1)$ We can see that $a^{\perp}=\{c,d\}$, $b^{\perp}=\{d\}$, $c^{\perp}=\{a\}$, $d^{\perp}=\{a,b\}$, 
$1^{\perp}=\emptyset$. 
Denote $A=\{a\}$, $B=\{d\}$, $C=\{a,b\}$, $D=\{c,d\}$; hence, 
$\mathcal{CL}(X^{\prime},\perp)=\{\emptyset, A, B, C, D, X^{\prime}\}$. \\
$(2)$ We have $b\wedge_Q c=0$ while $b\not\perp c$; therefore, the condition for $X$ to be an i-OML in 
Proposition \ref{ooml-50} is essential. \\
$(3)$ Since $a\perp d$ and $a\wedge_Q d^*=b\neq a$, this confirms Proposition \ref{ooml-40}. \\
$(4)$ Consider the structure $(\mathcal{CL}(X^{\prime},\perp),\ra,^*,X^{\prime})$ with the operations $\ra$, $\wedge_Q$ 
given below. 
\[
\begin{array}{c|cccccc}
\ra       & \emptyset  & A          & B          & C          & D          & X^{\prime} \\ \hline
\emptyset & X^{\prime} & X^{\prime} & X^{\prime} & X^{\prime} & X^{\prime} & X^{\prime} \\
A         & D          & X^{\prime} & D          & X^{\prime} & D          & X^{\prime} \\
B         & C          & C          & X^{\prime} & C          & X^{\prime} & X^{\prime} \\
C         & B          & X^{\prime} & B          & X^{\prime} & D          & X^{\prime} \\
D         & A          & A          & X^{\prime} & C          & X^{\prime} & X^{\prime} \\
X^{\prime} & \emptyset & A & B & C & D & X^{\prime}
\end{array}
\hspace{10mm}
\begin{array}{c|ccccccc}
\wedge_Q  & \emptyset  & A          & B          & C          & D          & X^{\prime} \\ \hline
\emptyset & \emptyset  & \emptyset  & \emptyset  & \emptyset  & \emptyset  & \emptyset  \\
A         & \emptyset  & A          & \emptyset  & C          & \emptyset  & A          \\
B         & \emptyset  & \emptyset  & B          & \emptyset  & D          & B          \\
C         & \emptyset  & A          & \emptyset  & C          & \emptyset  & C          \\
D         & \emptyset  & \emptyset  & B          & \emptyset  & D          & D          \\
X^{\prime} & \emptyset & A          & B          & C          & D          & X^{\prime}
\end{array}
.
\]

The structure $(\mathcal{CL}(X^{\prime},\perp),\ra,^*,X^{\prime})$ is an i-OL, but not an i-OML. 
Indeed, since $A\wedge_Q (B\ra A)=C\neq A$, condition $(IOM)$ is not satisfied. 
\end{example}

$\vspace*{1.0mm}$

\section{Sasaki projections on implicative-ortholattices} 

Sasaki projections were introduced by U. Sasaki in \cite{Sasaki} as projections on the lattice $L(H)$ of closed  subspaces of a Hilbert space $H$, mapping each subspace to a certain segment of that lattice. 
In quantum logic, the Sasaki projection and its dual can serve as the logical connectives, such as conjunction and implication (see for example \cite{Chajda1}). 
In this section, we adapt the notion of Sasaki projections to the case of implicative-ortholattices and 
study their properties. Using the orthogonality relation and the Sasaki projections on implicative-ortholattices, 
we prove a characterization theorem for implicative-orthomodular lattices. 
We mention that Sasaki projections play an important role in the definition and study of adjointness in 
implicative-ortholattices. 

\begin{definition} \label{sioml-10} 
\emph{
Let $X$ be an implicative-ortholattice and let $a\in X$. A \emph{Sasaki projection} on $X$ is a map 
$\varphi_a:X\longrightarrow X$ defined by: for all $x\in X$, $\varphi_a(x)=x\wedge_Q a$. 
}
\end{definition}

Denote $\mathcal{SP}(X)=\{\varphi_a\mid a\in X\}$ the set of all Sasaki projections on $X$. 
We write $\varphi_a x$, $\varphi_a\varphi_b x$ and $\varphi^*_a x$, instead of $\varphi_a(x)$,  
$(\varphi_a\circ \varphi_b)(x)$ and $(\varphi_a x)^*$, respectively. 

\begin{proposition} \label{sioml-20} 
Let $X$ be an implicative-ortholattice and let $a\in X$. The following hold for all $x\in X$: \\
$(1)$ $\varphi_x x=x$, $\varphi_1 x=\varphi_x 1=x$, $\varphi_0 x=\varphi_x 0=0$, $\varphi_x x^*=\varphi_{x^*} x=0;$ \\
$(2)$ $a\le_L x$ implies $\varphi_a x=a;$ 
$(3)$ $\varphi_{\varphi_a x}x=\varphi_a x;$ \\
$(4)$ $a\le_Q x$ implies $\varphi_x a=a;$ 
$(5)$ $\varphi_a$ is monotone.  
\end{proposition}
\begin{proof} 
$(1)$, $(2)$ and $(4)$ are straightforward. \\
$(3)$ We have: \\
$\hspace*{2.00cm}$ $\varphi_{\varphi_a x}x=((x^*\ra \varphi^*_a x)\ra \varphi^*_a x)^*
                   =((x^*\ra (x\wedge_Q a)^*)\ra \varphi^*_a x)^*$ \\
$\hspace*{3.20cm}$ $=((x^*\ra ((x^*\ra a^*)\ra a^*))\ra \varphi^*_a x)^*$ \\ 
$\hspace*{3.20cm}$ $=(((x^*\ra a^*)\ra (x^*\ra a^*))\ra \varphi^*_a x)^*$ \\
$\hspace*{3.20cm}$ $=(1\ra \varphi^*_a x)^*=\varphi_a x$. \\
$(5)$ Let $x_1,x_2\in X$ such that $x_1\le_L x_2$. 
Since by Lemma \ref{iol-30-20}$(9)$, $x_1\wedge_Q a\le_L x_2\wedge_Q a$, we have 
$\varphi_a x_1\le_L \varphi_a x_2$. Hence, $\varphi_a$ is monotone. 
\end{proof}

\begin{proposition} \label{sioml-30} Let $X$ be an implicative-orthomodular lattice. 
The following hold for all $a,b,x,y\in X$: \\
$(1)$ if $\varphi_a=\varphi_b=\varphi_1$, then $\varphi_{a\wedge_Q b}=\varphi_1;$   
$(2)$ $\varphi_a\varphi_a x=\varphi_a x;$ \\
$(3)$ $\varphi_a\varphi^*_a x=(a\ra x)^*;$  
$(4)$ $\varphi_a\varphi^*_a x\le_L x^*;$ \\
$(5)$ $\varphi_a x\le_L y^*$ iff $\varphi_a y\le_L x^*;$ 
$(6)$ $\varphi_a \varphi_b a=\varphi_a b;$ \\
$(7)$ if $a=a\wedge_Q b$, then $\varphi_a=\varphi_a\varphi_b$.  
\end{proposition}
\begin{proof}
$(1)$ From $\varphi_a=\varphi_b=\varphi_1$, we have $x\wedge_Q a=x\wedge_Q b=x$; thus, $x\le_Q a,b$, and hence $x\le_L a,b$.  
Applying Proposition \ref{ioml-50-10}$(1)$, we get $x\le_L a\wedge_Q b$; hence, $x\wedge_Q (a\wedge_Q b)=x$ for all $x\in X$. 
It follows that $\varphi_{a\wedge_Q b}=\varphi_1$. \\  
$(2)$ We have: \\
$\hspace*{2.00cm}$ $\varphi_a\varphi_a x=(x\wedge_Q a)\wedge_Q a=(((x\wedge_Q a)^*\ra a^*)\ra a^*)^*$ \\
$\hspace*{3.20cm}$ $=((a\ra (x\wedge_Q a))\ra a^*)^*=((a\ra x)\ra a^*)^*$ (Prop. \ref{ioml-50}$(1)$) \\
$\hspace*{3.20cm}$ $=((x^*\ra a^*)\ra a^*)^*=x\wedge_Q a=\varphi_a x$. \\
$(3)$ Similarly we get: \\
$\hspace*{2.00cm}$ $\varphi_a(\varphi_a x)^*=(x\wedge_Q a)^*\wedge_Q a=(((x\wedge_Q a)\ra a^*)\ra a^*)^*$ \\
$\hspace*{3.55cm}$ $=((a\ra (x\wedge_Q a)^*)\ra a^*)^*=((a\ra ((x^*\ra a^*)\ra a^*))\ra a^*)^*$ \\
$\hspace*{3.55cm}$ $=(((x^*\ra a^*)\ra (a\ra a^*))\ra a^*)^*=((x\wedge_Q a)^*\ra a^*)^*$ (by $(iG)$) \\
$\hspace*{3.55cm}$ $=(a\ra (x\wedge_Q a))^*=(a\ra x)^*$ (Prop. \ref{ioml-50}$(1)$). \\
$(4)$ It follows from $(3)$, since $x\le_L a\ra x$ implies $(a\ra x)^*\le_L x^*$. \\
$(5)$ From $\varphi_a x\le_L y^*$, we have $y\le_L (\varphi_a x)^*$. 
Since $\varphi_a$ is monotone, and applying $(4)$, we get $\varphi_a y\le_L \varphi_a(\varphi_a x)^*\le_L x^*$. 
The converse follows similarly. \\
$(6)$ Applying Proposition \ref{ioml-50-10}$(5)$, we get: \\
$\hspace*{2.00cm}$ $\varphi_a \varphi_b a=(\varphi_b a)\wedge_Q a=(a\wedge_Q b)\wedge_Q a=(((a\wedge_Q b)^*\ra a^*)\ra a^*)^*$ \\
$\hspace*{3.15cm}$ $=((a\ra (a\wedge_Q b))\ra a^*)^*=((a\ra b)\ra a^*)^*$ \\
$\hspace*{3.15cm}$ $=((b^*\ra a^*)\ra a^*)^*=b\wedge_Q a=\varphi_a b$. \\
$(7)$ If $a=a\wedge_Q b$, then $a\le_Q b$; hence, $a\le_L b$. By Proposition \ref{ioml-50-10}$(2)$, we have 
$x\wedge_Q a=(x\wedge_Q b)\wedge_Q a$ for all $x\in X$; thus, $\varphi_a=\varphi_a\varphi_b$. 
\end{proof}

\begin{proposition} \label{sioml-40} Let $X$ be an implicative-orthomodular lattice. 
The following hold for all $a,b,x,y\in X$: \\
$(1)$ $\varphi_a x=x$ iff $x\le_L a;$ 
$(2)$ $\varphi_a x=0$ iff $x\le_L a^*;$ \\
$(3)$ if $a\le_L b$, then $\varphi_a\varphi_b x=\varphi_a x;$ 
$(4)$ $\varphi^*_a x=\varphi_a x\ra y$ iff $\varphi^*_a y=\varphi_a y\ra x;$ \\
$(5)$ $\varphi_a^2=\varphi_0$ iff $\varphi_a 1\le_L \varphi_a^*1;$ 
$(6)$ $\varphi_a x\perp y$ iff $x\perp \varphi_a y;$ 
$(7)$ $x\perp y$ iff $\varphi_xy=0;$ \\
$(8)$ if $x\perp a$, then $\varphi_ax\perp a^*$.        
\end{proposition}
\begin{proof}
$(1)$ If $\varphi_a x=x$, then $x\wedge_Q a=x$, and hence $x\le_Q a$; thus, $x\le_L a$. 
Conversely, if $x\le_L a$, then $x=(x\ra a^*)^*$ and applying Proposition \ref{ioml-50}$(9)$, we have: 
$\varphi_a x=x\wedge_Q a=((x^*\ra a^*)\ra a^*)^*=(((x\ra a^*)\ra a^*)\ra a^*)^*=((x\vee_Q a^*)\ra a^*)^*=(x\ra a^*)^*=x$. \\
$(2)$ Assume that $\varphi_a x=0$, that is, $x\wedge_Q a=0$.  It follows that $(x^*\ra a^*)\ra a^*=1$; hence, 
$x^*\ra a^*\le a^*$. Since $a^*\le_L x^*\ra a^*$, by Theorem \ref{dioml-05-60} we get $x^*\ra a^*=a^*$. 
It follows that $a=(a\ra x)^*$, that is, $a\le_L x^*$; hence, $x\le_L a^*$. 
Conversely, $x\le_L a^*$ implies $x=(x\ra a)^*$, thus, $x^*=x\ra a=a^*\ra x^*$. 
Hence, by $(impl)$, $\varphi_a x=x\wedge_Q a=((x^*\ra a^*)\ra a^*)^*=(((a^*\ra x^*)\ra a^*)\ra a^*)^*=(a^*\ra a^*)^*=0$. \\
$(3)$ If $a\le_L b$, then $a^*=a\ra b^*$ and we have: 
$\varphi_a\varphi_b x=(x\wedge_Q b)\wedge_Q a=(((x\wedge_Q b)^*\ra a^*)\ra a^*)^*$. 
Using Proposition \ref{ioml-50}$(9)$, we get: 
$(x\wedge_Q b)^*\ra a^*=(x^*\vee_Q b^*)\ra (a\ra b^*)=a\ra ((x^*\vee_Q b^*)\ra b^*)=a\ra (x^*\ra b^*)=x^*\ra (a\ra b^*)=
x^*\ra a^*$. It follows that $\varphi_a\varphi_b x=((x^*\ra a^*)\ra a^*)^*=x\wedge_Q a=\varphi_a x$. \\
$(4)$ If $\varphi^*_a x=\varphi_a x\ra y$, then $\varphi_a x\le_L y^*$; hence, $y\le_L \varphi^*_a x$. 
Applying Propositions \ref{sioml-20}$(5)$, \ref{sioml-30}$(4)$, we get 
$\varphi_a y\le_L \varphi_a\varphi^*_a x\le_L x^*$; hence, $\varphi^*_a y=\varphi_a y\ra x$. 
The converse follows similarly. \\
$(5)$ If $\varphi_a^2=\varphi_0$, then $\varphi_a(\varphi_a 1)=0$, and by $(2)$, $\varphi_a 1\le_L a^*=\varphi_a^*1$. 
Conversely, assume that $\varphi_a 1\le_L \varphi_a^*1$. 
Since $x\le_L 1$, applying Proposition \ref{sioml-30}$(4)$, we get 
$\varphi_a^2 x\le_L\varphi_a^2 1=\varphi_a(\varphi_a 1)\le_L \varphi_a(\varphi_a^*1)\le_L 1^*=0$.  
It follows that $\varphi_a^2 x=0$ for all $x\in X$. Hence, $\varphi_a^2=\varphi_0$. \\
$(6)$ It is a consequence of $(4)$. \\
$(7)$ It follows from Proposition \ref{ooml-50}. \\
$(8)$ Since $x\perp a$ implies $x^*=x\ra a$, it follows that: \\
$\hspace*{2.00cm}$ $a^*\ra \varphi_a x=a^*\ra \varphi_a(x\ra a)^*=a^*\ra ((x\ra a)^*\wedge_Q a)$ \\
$\hspace*{3.70cm}$ $=a^*\ra (((x\ra a)\ra a^*)\ra a^*)^*=a^*\ra (((a^*\ra x^*)\ra a^*)\ra a^*)^*$ \\
$\hspace*{3.70cm}$ $=a^*\ra (a^*\ra a^*)^*=a$ (by $(impl)$). \\
Hence, $a^*\perp \varphi_ax$, that is, $\varphi_ax \perp a^*$. 
\end{proof}

In the next result, by providing a characterization theorem of implicative-orthomodular lattices, we show how 
orthogonality contributes to the establishment of certain features of quantum logic. 

\begin{theorem} \label{sioml-40-10} An implicative-ortholattice $X$ is an implicative-orthomodular lattice if and  
only if $\varphi_yx\perp z$ implies $x\perp \varphi_y z$ for all $x, y, z\in X$. 
\end{theorem}
\begin{proof}
If $X$ is an i-OML and $x,y,z\in X$, then by Proposition \ref{sioml-40}$(6)$, 
$\varphi_yx\perp z$ implies $x\perp \varphi_y z$. 
Conversely, let $X$ be an i-OL such that $\varphi_yx\perp z$ implies $x\perp \varphi_y z$ for all $x, y, z\in X$. 
Let $x\le_L y$, that is, $x=(x\ra y^*)^*$. 
It follows that: \\
$\hspace*{2.00cm}$ $\varphi_{x^*}y=y\wedge_Q x^*=((y^*\ra x)\ra x)^*=((y^*\ra (x\ra y^*)^*)\ra x)^*$ \\
$\hspace*{3.00cm}$ $=(((y\ra x^*)\ra y)\ra x)^*=(y\ra x)^*$ (by $(impl)$), \\ 
$\hspace*{2.00cm}$ $\varphi_{x^*}y\ra (y\ra x)=(y\ra x)^*\ra (y\ra x)=y\ra x=\varphi_{x^*}^*y$ (by $(iG)$). \\
Hence, $\varphi_{x^*}y\perp (y\ra x)$, and by hypothesis, $\varphi_{x^*}(y\ra x)\perp y$, that is, 
$\varphi_{x^*}^*(y\ra x)=\varphi_{x^*}(y\ra x)\ra y$. 
It follows that: \\
$\hspace*{1.20cm}$ $y\le_L \varphi_{x^*}(y\ra x)\ra y=\varphi_{x^*}^*(y\ra x)=((y\ra x)\wedge_Q x^*)^*=(y\ra x)^*\vee_Q x$ \\
$\hspace*{5.00cm}$ $=((y\ra x)^*\ra x)\ra x=(x^*\ra (x^*\ra y^*))\ra x$ \\
$\hspace*{5.00cm}$ $=(x^*\ra y^*)\ra x$ (by $(pi)$) \\
$\hspace*{5.00cm}$ $=(y\ra x)\ra x=y\vee_Q x$. \\
On the other hand, from $x\le_L y$ and using $(impl)$, we obtain $y\vee_Q x=(y\ra x)\ra x\le_L(y\ra x)\ra y=y$. 
Hence, $x\le_L y$ implies $y=y\vee_Q x$, and by Theorem \ref{dioml-05-50}, $X$ is an i-OML. 
\end{proof}

A \emph{projection} is a map $f$ from a set $X$ to $X$ satisfying $f\circ f=f$ (see for example \cite{Chajda1}). 
Given an implicative-orthomodular lattice $X$, by Propositions \ref{sioml-20}, \ref{sioml-30}, \ref{sioml-40} 
a Sasaki projection $\varphi_a:X\longrightarrow X$ is a monotone projection from $X$ onto $[0, a]$. 
Moreover, the \emph{dual} $\bar{\varphi}_a$ of $\varphi_a$ defined by 
$\bar{\varphi}_a x=(\varphi_a x^*)^*=x\vee_Q a^*$ is a monotone projection from $X$ onto $[a^*, 1]$. \\

Given an orthomodular lattice $(X,\wedge,\vee,^{\prime},0,1)$ and $x, y\in X$, according to \cite{Holl1} and \cite{Foulis2}, $x$ \emph{commutes} with $y$ if $(x\wedge y)\vee (x\wedge y^{\prime})=x$. 
Using the commutativity between two elements of an orthomodular lattice $X$, D.J. Foulis \cite{Foulis2} and 
S.S. Holland \cite{Holl1}, independently, proved that, given a triple of elements of $X$, if at least one 
of them commutes with the other two, then the triple is distributive. 
Based on Sasaki projections, we define the commutativity $\mathcal{C}$ between two elements of implicative-ortholattices, and we study their properties. We show that an implicative-ortholattice $X$ is an implicative-orthomodular lattice if and only if for all $x,y\in X$, $x\mathcal{C} y$ implies $y\mathcal{C} x$. 
Moreover, we prove that the composition of two Sasaki projections $\varphi_a$, $\varphi_b$ on an implicative-orthomodular lattice is commutative if and only if $a$ commutes with $b$. 
Commutativity between two elements plays an crucial role in the study of distributivity of an implicative-orthomodular 
lattice: if any two elements commute, then the implicative-orthomodular lattice is distributive; that is, it 
becomes an implicative-Boolean algebra. In other words, in this case, the quantum logic becomes classical logic. 

\begin{definition} \label{cioml-10} 
\emph{
Let $X$ be an implicative-ortholattice and let $x,y\in X$. We say that $x$ \emph{commutes} with $y$, 
denoted by $x\mathcal{C} y$, if $\varphi_x y=(x\ra y^*)^*$. 
}
\end{definition}

\begin{lemma} \label{cioml-20} Let $X$ be an implicative-ortholattice. 
The following hold for all $x,y\in X$: \\
$(1)$ $x\mathcal{C} x$, $x\mathcal{C} 0$, $0\mathcal{C} x$, $x\mathcal{C} 1$, $1\mathcal{C} x$, 
$x\mathcal{C} x^*$, $x^*\mathcal{C} x;$ \\
$(2)$ if $x\le_L y$ or $x\le_L y^*$, then $x\mathcal{C} y;$ \\
$(3)$ $x\mathcal{C} (y\ra x)$, $x\mathcal{C} (x^*\ra y)$ and $y\mathcal{C} (x^*\ra y)$.
\end{lemma}
\begin{proof}
$(1)$ It follows from Proposition \ref{sioml-20}$(1)$. \\ 
$(2)$ Since $x\le_L y$ implies $x=(x\ra y^*)^*$ and $x\le y$, we have 
$\varphi_x y=y\wedge_Q x=((y^*\ra x^*)\ra x^*)^*=((x\ra y)\ra x^*)^*=(1\ra x^*)^*=x=(x\ra y^*)^*$, 
that is, $x\mathcal{C} y$. 
Similarly, from $x\le_L y^*$ we have $x^*=x\ra y$ and $x\le_L y^*$; hence, $x\ra y^*=1$. It follows that  
$\varphi_x y=y\wedge_Q x=((y^*\ra x^*)\ra x^*)^*=((x\ra y)\ra x^*)^*=(x^*\ra x^*)^*=0=(x\ra y^*)^*$; thus,   
$x\mathcal{C} y$. \\
$(3)$ It follows from $(2)$, since $x\le_L y\ra x$ and $x, y\le_L x^*\ra y$.  
\end{proof}

The next theorem provides a characterization of implicative-orthomodular lattices in terms of the commutativity 
between two elements. This characterization is used to prove several important results in this section.

\begin{theorem} \label{cioml-30} Let $X$ be an implicative-ortholattice. The following are equivalent: \\
$(a)$ $X$ is an implicative-orthomodular lattice; \\
$(b)$ for all $x, y\in X$, $x\mathcal{C} y$ implies $y\mathcal{C} x$. 
\end{theorem}
\begin{proof}
$(a)\Rightarrow (b)$ Suppose $X$ is an i-OML and let $x,y\in X$ such that $x\mathcal{C} y$, that is, 
$\varphi_x y=y\wedge_Q x=(x\ra y^*)^*$. 
By Proposition \ref{ioml-50}$(7)$, $(x\wedge_Q y)\ra (y\wedge_Q x)=1$; hence, $(x\wedge_Q y)\ra (x\ra y^*)^*=1$;  
therefore $x\wedge_Q y\le (x\ra y^*)^*$. 
On the other hand, by $(IOM)$, $x\le_Q x^*\ra y^*$; thus, $(x^*\ra y^*)\ra y^*\le_Q x\ra y^*$. 
It follows that $(x\ra y^*)^*\le_L x\wedge_Q y$, and applying Proposition \ref{ioml-50}$(5)$, we get 
$x\wedge_Q y=(x\ra y^*)^*$. Hence, $\varphi_y x=(x\ra y^*)^*=(y\ra x^*)^*$, that is, $y\mathcal{C} x$. \\
$(b)\Rightarrow (a)$ Let $x,y\in X$ such that $x\le_L y$, that is, $(x\ra y^*)^*=x$. 
By Lemma \ref{cioml-20}$(2)$, we have $x\mathcal{C} y$, and by $(b)$, $y\mathcal{C} x$. 
Hence, $\varphi_y x=(y\ra x^*)^*=(x\ra y^*)^*=x$; thus, $x\wedge_Q y=x$, that is, $x\le_Q y$. 
Thus, $x\le_L y$ implies $x\le_Q y$, and by Theorem \ref{dioml-05-50}, $X$ is an i-OML. 
\end{proof}

\begin{corollary} \label{cioml-40} Let $X$ be an implicative-ortholattice. 
Then $X$ is an implicative-orthomodular lattice if and only if for all $x, y\in X$, $x\mathcal{C} y$ 
implies $x\wedge_Q y=y\wedge_Q x$. 
\end{corollary}
\begin{proof}
Let $X$ be an i-OML and let $x,y\in X$ such that $x\mathcal{C} y$, that is, $y\wedge_Q x=(x\ra y^*)^*$. 
By Theorem \ref{cioml-30}, $x\mathcal{C} y$ implies $y\mathcal{C} x$; hence, 
$x\wedge_Q y=(y\ra x^*)^*=(x\ra y^*)^*=y\wedge_Q x$. 
Conversely, let $X$ be an i-OL and let $x,y\in X$ such that $x\mathcal{C} y$, that is,  
$y\wedge_Q x=(x\ra y^*)^*$. By hypothesis, $y\wedge_Q x=x\wedge_Q y$; thus, $x\wedge_Q y=(x\ra y^*)^*=(y\ra x^*)^*$. 
It follows that $\varphi_y x=(y\ra x^*)^*$; hence, $y\mathcal{C} x$, and by Theorem \ref{cioml-30}, $X$ is an i-OML. 
\end{proof}

\begin{lemma} \label{cioml-50} Let $X$ be an implicative-orthomodular lattice and let $x,y \in X$ such that  $x\mathcal{C} y$. Then $x\mathcal{C} y^*$, $x^*\mathcal{C} y$ and $x^*\mathcal{C} y^*$.  
\end{lemma}
\begin{proof}
Assume that $x\mathcal{C} y$, that is, $\varphi_x y=y\wedge_Q x=(x\ra y^*)^*$. Using Proposition \ref{ioml-50}$(1)$, 
we have: \\
$\hspace*{2.10cm}$ $\varphi_x y^*=y^*\wedge_Q x=((y\ra x^*)\ra x^*)^*=((x\ra y^*)\ra x^*)^*$ \\
$\hspace*{3.00cm}$ $=((y\wedge_Q x)^*\ra x^*)^*=(x\ra (y\wedge_Q x))^*$ \\
$\hspace*{3.00cm}$ $=(x\ra y)^*=(x\ra (y^*)^*)^*$. \\
Hence, $x\mathcal{C} y^*$. By Theorem \ref{cioml-30}, we get $y^*\mathcal{C} x$; thus, $y^*\mathcal{C} x^*$. 
Applying again Theorem \ref{cioml-30}, we have $x^*\mathcal{C} y^*$, and finally $x^*\mathcal{C} y$. 
\end{proof}

\begin{proposition} \label{cioml-60} Let $X$ be an implicative-orthomodular lattice and let $x,y\in X$. 
Then $x\mathcal{C} y$ if and only if $(x\ra y^*)\ra (x\ra y)^*=x$. 
\end{proposition}
\begin{proof}
If $x\mathcal{C} y$, then $\varphi_x y=y\wedge_Q x=(x\ra y^*)^*$. Using Proposition \ref{ioml-50}$(4)$, we have 
$(x\ra y^*)\ra (x\ra y)^*=(y\wedge_Q x)^*\ra (x\ra y)^*=(x\ra y)\ra (y\wedge_Q x)=x$. 
Conversely, suppose $(x\ra y^*)\ra (x\ra y)^*=x$. Then we have: \\
$\hspace*{2.00cm}$ $x\wedge_Q y=((x^*\ra y^*)\ra y^*)^*=((y\ra x)\ra y^*)^*$ \\
$\hspace*{3.00cm}$ $=((y\ra ((x\ra y^*)\ra (x\ra y)^*))\ra y^*)^*$ \\
$\hspace*{3.00cm}$ $=(((x\ra y^*)\ra (y\ra (x\ra y)^*))\ra y^*)^*$ \\
$\hspace*{3.00cm}$ $=(((x\ra y^*)\ra ((x\ra y)\ra y^*))\ra y^*)^*$ \\
$\hspace*{3.00cm}$ $=(((x\ra y^*)\ra ((y^*\ra x^*)\ra y^*))\ra y^*)^*$ \\
$\hspace*{3.00cm}$ $=(((x\ra y^*)\ra y^*)\ra y^*)^*$ (by $(impl)$) \\
$\hspace*{3.00cm}$ $=((x\vee_Q y^*)\ra y^*)^*=(y\ra (x^*\wedge_Q y))^*$ \\
$\hspace*{3.00cm}$ $=(y\ra x^*)^*$ (by Prop. \ref{ioml-50}$(1)$). \\
Hence, $\varphi_y x=(y\ra x^*)^*$, that is, $y\mathcal{C} x$, and by Theorem \ref{cioml-30}, $x\mathcal{C} y$. 
\end{proof}

Recall that in \cite{Ciu84} the commutativity relation between two elements of an implicative-orthomodular 
lattice was defined based on the formula given in Proposition \ref{cioml-60}. 
This formula is useful for an easier proof of some results in the paper. 

\begin{remark} \label{cioml-60-05} Let $X$ be an implicative-orthomodular lattice and let $x,y \in X$. 
Then $x\perp y$ implies $x\mathcal{C} y$. 
Indeed, if $x\perp y$, then $y\perp x$; hence, $x^*=x\ra y$ and $y^*=y\ra x$. 
It follows that $(x\ra y^*)\ra (x\ra y)^*=(x\ra (y\ra x))\ra (x^*)^*=1\ra x=x$, and according to 
Proposition \ref{cioml-60} we get $x\mathcal{C} y$. 
As we will see in Example \ref{cioml-60-70}, the converse is not true. 
\end{remark}

\begin{proposition} \label{cioml-60-10} Let $X$ be an implicative-orthomodular lattice and let $x,y \in X$. 
Then $x\mathcal{C} y$ if and only if $x\wedge_Q y=(x\ra y^*)^*$. 
\end{proposition}
\begin{proof}
Suppose $x\mathcal{C} y$. Then, by Theorem \ref{cioml-30}, $y\mathcal{C} x$. 
Hence, $(x\ra y^*)^*=(y\ra x^*)^*=\varphi_y x=x\wedge_Q y$.  
Conversely, if $x\wedge_Q y=(x\ra y^*)^*$, using Proposition \ref{ioml-50}$(4)$, we get: \\
$\hspace*{2.00cm}$ $(y\ra x^*)\ra (y\ra x)^*=(y\ra x)\ra (x\ra y^*)^*=(y\ra x)\ra (x\wedge_Q y)=y$. \\
By Proposition \ref{cioml-60}, it follows that $y\mathcal{C} x$. Hence, by Theorem \ref{cioml-30}, $x\mathcal{C} y$. 
\end{proof}

\begin{corollary} \label{cioml-60-20} Let $X$ be an implicative-orthomodular lattice and let $x,y \in X$. 
The following are equivalent: \\ 
$(a)$ $x\mathcal{C} y;$ \\ 
$(b)$ $x\wedge_Q y=y\wedge_Q x;$ \\
$(c)$ $x\vee_Q y=y\vee_Q x;$ \\
$(d)$ $\varphi_x y=\varphi_y x$.  
\end{corollary}
\begin{proof}
$(a)\Leftrightarrow (b)$ By Theorem \ref{cioml-30}, $x\mathcal{C} y$ iff $y\mathcal{C} x$.  
By Proposition \ref{cioml-60-10}, $x\mathcal{C} y$ iff $x\wedge_Q y=y\wedge_Q x$, since $(x\ra y^*)^*=(y\ra x^*)^*$. \\
$(a)\Leftrightarrow (c)$ By Lemma \ref{cioml-50} and taking into consideration that $(a)\Leftrightarrow (b)$,  we have $x\mathcal{C} y$ iff $x^*\mathcal{C} y^*$ iff $x^*\wedge_Q y^*=y^*\wedge_Q x^*$ iff $(x^*\wedge_Q y^*)^*=(y^*\wedge_Q x^*)^*$ iff $x\vee_Q y=y\vee_Q x$. \\ 
$(b)\Leftrightarrow (d)$ It is obvious, since $\varphi_x y=\varphi_y x$ iff $x\wedge_Q y=y\wedge_Q x$. 
\end{proof} 

In what follows, inspired by \cite[Thm. 5.2]{Sasaki}, we give a characterization of the commutativity relation 
between two elements by means of Sasaki projections. 

\begin{theorem} \label{cioml-60-40} Let $X$ be an implicative-orthomodular lattice and let $a,b \in X$. 
The following are equivalent: \\
$(a)$ $a\mathcal{C} b;$ \\ 
$(b)$ $\varphi_a\varphi_b x=\varphi_b\varphi_a x=\varphi_{a\wedge_Q b}x$ for all $x\in X;$ \\
$(c)$ for all $x\in X$, $x\le_L a$ implies $\varphi_b x\le_L a$ and $x\le_L b$ implies $\varphi_a x\le_L b$.   
\end{theorem}
\begin{proof}
$(a)\Rightarrow (b)$ Let $a,b\in X$ such that $a\mathcal{C} b$; hence, $a\wedge_Q b=b\wedge_Q a$. 
Let $x\in X$. By Proposition \ref{ioml-50}$(10)$, $(x\wedge_Q b)\wedge_Q a=(x\wedge_Q b)\wedge_Q (b\wedge_Q a)$; hence,  
$\varphi_a\varphi_b x=\varphi_{b\wedge_Q a}\varphi_b x=\varphi_{a\wedge_Q b}\varphi_b x$. 
Since $a\wedge_Q b\le_L b$, applying Proposition \ref{sioml-40}$(3)$, $\varphi_{a\wedge_Q b}\varphi_b x=\varphi_{a\wedge_Q b}x$. 
Hence, $\varphi_a\varphi_b x=\varphi_{a\wedge_Q b}x$, and similarly $\varphi_b\varphi_a x=\varphi_{a\wedge_Q b}x$. \\
$(b)\Rightarrow (c)$ Taking $x:=1$ in $(b)$, we get $a\wedge_Q b=b\wedge_Q a$. 
Let $x\in X$ such that $x\le_L a$; hence, $x\le_Q a$. 
It follows that $\varphi_a x=x\wedge_Q a=x$, and by $(b)$, we have: 
$\varphi_b x=\varphi_b\varphi_a x=\varphi_{a\wedge_Q b}x=\varphi_{b\wedge_Q a}x=x\wedge_Q (b\wedge_Q a)\le_L b\wedge_Q a\le_L a$. 
Similarly, $x\le_L b$ implies $x=\varphi_b x$; thus,  
$\varphi_a x=\varphi_a\varphi_b x=\varphi_{a\wedge_Q b}x=x\wedge_Q (a\wedge_Q b)\le_L a\wedge_Q b\le_L b$. \\
$(c)\Rightarrow (a)$ For $x:=a$ in the first implication of $(c)$, we get $\varphi_b a\le_L a$; 
hence, $a\wedge_Q b\le_L a$. 
Since $a\wedge_Q b\le_L b$, applying Proposition \ref{ioml-50-10}$(1)$ we have $a\wedge_Q b\le_L b\wedge_Q a$. 
Similarly, taking $x:=b$ in the second implication, it follows that $b\wedge_Q a\le_L b$. Since $b\wedge_Q a\le_L a$, by 
Proposition \ref{ioml-50-10}$(1)$ we get $b\wedge_Q a\le_L a\wedge_Q b$. Hence, $a\wedge_Q b=b\wedge_Q a$; thus,  
by Corollary \ref{cioml-60-20}, $a\mathcal{C} b$. 
\end{proof}

\begin{corollary} \label{cioml-60-50} The composition of two Sasaki projections $\varphi_a$, $\varphi_b$ on an 
implicative-orthomodular lattice is commutative if and only if $a\mathcal{C} b$.    
\end{corollary}

\begin{example} \label{cioml-60-60}
Consider the i-OL $(X,\ra,^*,1)$ from Example \ref{dioml-60}, where $X=\{0,a,b,c,d,1\}$. 
We have shown that $X$ is not an i-OML. \\
$(1)$ We have $\varphi_a b\perp c$, but $b\not\perp \varphi_a c$. It follows that in Theorem \ref{sioml-40-10} 
the condition for $X$ to be an i-OML is necessary. \\
$(2)$ Since $\varphi_a b=b\wedge_Q a=a=(a\ra b^*)^*$, it follows that $a\mathcal{C} b$. 
However, $\varphi_b a=a\wedge_Q b=b\neq a=(b\ra a^*)^*$, therefore $b$ does not commute with $a$. 
This confirms the results of Theorem \ref{cioml-30} and Corollary \ref{cioml-60-20}. \\
$(3)$ We also have: $\varphi_a\varphi_ba=a$, $\varphi_b\varphi_aa=b$, $\varphi_{a\wedge_Q b}a=b$. 
Hence, the Sasaki projections $\varphi_a$ and $\varphi_b$ do not commute, even though $a\mathcal{C} b$. 
This means that the condition on $X$ in Theorem \ref{cioml-60-40} to be an i-OML is essential. 
\end{example}

\begin{example} \label{cioml-60-70}
Let $(X,\ra,^*,1)$ be the i-OML given in Example \ref{dioml-60-10}, where \\
$X=\{0,a,b,c,d,e,f,g,h,1\}$. \\
$(1)$ We have $a\mathcal{C} f$, but $a\not\perp f$; hence the converse of remark \ref{cioml-60-05} is not true. \\
$(2)$ Since $a\wedge_Q e=e\wedge_Q a (=e)$, it follows that $a\mathcal{C} e$. Using the table   
\[
\begin{array}{c|cccccccccc}
x                                   & 0 & a & b & c & d & e & f & g & h & 1 \\ \hline
\varphi_a\varphi_ex=(x\wedge_Q e)\wedge_Q a & 0 & e & 0 & e & e & e & 0 & 0 & e & e \\ \hline
\varphi_e\varphi_ax=(x\wedge_Q a)\wedge_Q e & 0 & e & 0 & e & e & e & 0 & 0 & e & e \\ \hline
\varphi_{a\wedge_Q e}x=x\wedge_Q (a\wedge_Q e)  & 0 & e & 0 & e & e & e & 0 & 0 & e & e \\ \hline
\end{array}
,
\]

we can see that $\varphi_a\varphi_ex=\varphi_e\varphi_ax=\varphi_{a\wedge_Q e}x$ for all $x\in X$, 
that is, $\varphi_a\varphi_e=\varphi_e\varphi_a$. This is consistent with Theorem \ref{cioml-60-40}. 
\end{example}

$\vspace*{1mm}$

\section{Implicative-Boolean center} 

In this section, based on the divisibility relation on implicative-ortholattices, we provide certain characterizations of implicative-Boolean algebras. We also introduce the notion of the center $\mathcal{CC}(X)$ of an implicative-orthomodular lattice $X$, proving that $\mathcal{CC}(X)$ is an implicative-Boolean subalgebra of $X$. 
Moreover, we prove that an implicative-ortholattice $X$ is an implicative-orthomodular lattice 
if and only if every two orthogonal elements of $X$ are contained in an implicative-Boolean subalgebra of $X$. \\
From the viewpoint of quantum logic, an implicative-orthomodular lattice represents the lattice of propositions 
of a quantum system, while its center corresponds to propositions that commute with the entire structure. 
In other words, the implicative-Boolean center captures the maximal ``classical core" of a 
quantum structure represented by an implicative-orthomodular lattice; thus, the implicative-Boolean center provides 
a bridge between classical and quantum aspects of the structure. 
By Definition \ref{dioml-100}, an implicative-Boolean algebra is an implicative-ortholattice satisfying 
condition $(@)$. 

\begin{definition} \label{cioml-70}
\emph{
Given an implicative-ortholattice $X$, we define the \emph{divisibility relation} $\mathcal{D}$ on $X$ by: 
for all $x, y\in X$, $x\mathcal{D} y$ if and only if $x$ and $y$ satisfy condition $(Idiv)$.  
}
\end{definition}

\begin{lemma} \label{cioml-90} Let $X$ be an implicative-ortholattice and let $x,y\in X$.
Then $x\mathcal{C} y$ if and only if $x\mathcal{D} y$. 
\end{lemma}
\begin{proof}
We have $x\mathcal{C} y$ iff $\varphi_x y=(x\ra y^*)^*$ iff $(y^*\ra x^*)\ra x^*=x\ra y^*$ iff 
$(x\ra y)\ra x^*=x\ra y^*$ iff $x\ra (x\ra y)^*=x\ra y^*$ iff $x\mathcal{D} y$. 
\end{proof}

\begin{corollary} \label{cioml-90-10} An implicative-ortholattice $X$ is an implicative-orthomodular 
lattice if and only if $x\mathcal{D} y$ implies $y\mathcal{D} x$. 
\end{corollary}

\begin{lemma} \label{cioml-90-15} Let $X$ be an implicative-ortholattice. The following hold for all $x,y\in X$: \\
$(1)$ $x\mathcal{D} x$, $x\mathcal{D} 0$, $0\mathcal{D} x$, $x\mathcal{D} 1$, $1\mathcal{D} x$, 
$x\mathcal{D} x^*$, $x^*\mathcal{D} x;$ \\
$(2)$ if $x\le_L y$ or $x\le_L y^*$, then $x\mathcal{D} y;$ \\
$(3)$ $x\mathcal{D} (y\ra x)$, $x\mathcal{D} (x^*\ra y)$ and $y\mathcal{D} (x^*\ra y)$. \\
If $X$ is an implicative-orthomodular lattice, then: \\
$(4)$ if $x\perp y$, then $x\mathcal{D} y$, $y\mathcal{D} x$, $x\mathcal{D} y^*$, $y^*\mathcal{D} x;$ \\
$(5)$ $x^*\mathcal{D} (x^*\ra y)$, $y^*\mathcal{D} (x^*\ra y)$, $x\mathcal{D} (x^*\ra y)^*$, 
$y\mathcal{D} (x^*\ra y)^*$.  
\end{lemma} 
\begin{proof}
$(1)$-$(3)$ follow by Lemmas \ref{cioml-20}, \ref{cioml-90}. \\
$(4)$ follows by Remark \ref{cioml-60-05} and Lemma \ref{cioml-50}. \\
$(5)$ follows by $(2)$, $(3)$ and Lemma \ref{cioml-50}. 
\end{proof}

\begin{proposition} \label{cioml-90-20} Any implicative-Boolean algebra is an implicative-orthomodular lattice.  
\end{proposition}
\begin{proof}
Let $X$ be an i-Boolean algebra. 
By definition, $X$ is an i-OL. Let $x, y\in X$ such that $x\le_Q y$, hence, by Proposition \ref{qbe-20}$(4)$, 
$x\le y$, that is $x\ra y=1$. 
According to Theorem \ref{dioml-110}, $X$ satisfies condition 
$(Idis_1)$ $((x^*\ra y)\ra z^*)^*=(x\ra z^*)\ra (y\ra z^*)^*$, and we obtain: 
\begin{align*} 
y\vee_Q x
&=(y\ra x)\ra x=((y^*)^*\ra x)\ra (x^*)^*=((y^*\ra x)\ra (x\ra x)^*)^* \\
&=y^*\ra x=x^*\ra y=((x^*\ra y)\ra (y\ra y)^*)^* \\
&=((x^*)^*\ra y)\ra (y^*)^*=(x\ra y)\ra y=1\ra y=y. 
\end{align*}
By Theorem \ref{dioml-05-50}, $X$ is an i-OML. 
\end{proof}

The following results provide certain characterizations of implicative-Boolean algebras that show when  
the quantum logic becomes classical. These results help to understand how quantum logic generalizes and differs 
from classical logic. 

\begin{theorem} \label{cioml-100} Let $X$ be an implicative-orthomodular lattice. The following are equivalent: \\
$(a)$ $X$ is an implicative-Boolean algebra; \\
$(b)$ $x\wedge_Q y=(x\ra y^*)^*$ for all $x, y\in X;$ \\   
$(c)$ $x\wedge_Q y=y\wedge_Q x$ for all $x, y\in X;$ \\
$(d)$ $x\vee_Q y=y\vee_Q x$ for all $x, y\in X;$ \\
$(e)$ $x\mathcal{C} y=y\mathcal{C} x$ for all $x, y\in X;$ \\
$(f)$ $x\mathcal{D} y=y\mathcal{D} x$ for all $x, y\in X$. 
\end{theorem}
\begin{proof}
It follows from Lemma \ref{cioml-90}, Proposition \ref{cioml-60-10} and Corollaries \ref{cioml-60-20}, \ref{cioml-90-10}. 
\end{proof} 

\begin{theorem} \label{cioml-100-10} Let $X$ be an implicative-orthomodular lattice. The following are equivalent: \\
$(a)$ $X$ is an implicative-Boolean algebra; \\  
$(b)$ $x\wedge_Q y\le_L x$ for all $x, y\in X;$ \\  
$(c)$ $x\le_L x\vee_Q y$ for all $x, y\in X$.   
\end{theorem}
\begin{proof} 
$(a)\Rightarrow (b)$ If $X$ is an i-Boolean algebra and $x,y\in X$, then $x\mathcal{D} y$; thus, 
$x\mathcal{C} y$. It follows that $x\wedge_Q y=y\wedge_Q x\le_L x$. \\
$(b)\Rightarrow (a)$ Let $X$ be an i-OML such that $x\wedge_Q y\le_L x$ for all $x, y\in X$. 
Since $x\wedge_Q y\le_L y$, applying Proposition \ref{ioml-50-10}$(1)$, we get $x\wedge_Q y\le_L y\wedge_Q x$. 
Similarly, from $y\wedge_Q x\le_L x$ and $y\wedge_Q x\le_L y$, we have $y\wedge_Q x\le_L x\wedge_Q y$; hence, $x\wedge_Q y=y\wedge_Q x$. 
By Corollary \ref{cioml-60-20}, it follows that $x\mathcal{C} y$; therefore, by Lemma \ref{cioml-90}, $X$ is an i-Boolean algebra. \\
$(a)\Rightarrow (c)$ Since $X$ is an i-Boolean algebra and $x,y\in X$, we have $x\mathcal{D} y$.  
It follows that $x\mathcal{C} y$, and $x\le_L y\vee_Q x=x\vee_Q y$. \\
$(c)\Rightarrow (a)$ Let $X$ be an i-OML such that $x\le_L x\vee_Q y$ for all $x, y\in X$.
Since $y\le_L x\vee_Q y$, applying Proposition \ref{ioml-50-10}$(4)$, we get $y\vee_Q x\le_L x\vee_Q y$. 
Similarly, from $x\le_L y\vee_Q x$ and $y\le_L y\vee_Q x$, we have $x\vee_Q y\le_L y\vee_Q x$. 
Hence, $x\vee_Q y=y\vee_Q x$, and applying Corollary \ref{cioml-60-20}, we conclude that $X$ is an i-Boolean algebra. 
\end{proof}

\begin{theorem} \label{cioml-100-20} Let $X$ be an implicative-orthomodular lattice. The following are equivalent: \\
$(a)$ $X$ is an implicative-Boolean algebra; \\  
$(b)$ $\le$ $\subseteq$ $\le_L$.  
\end{theorem}
\begin{proof} 
$(a)\Rightarrow (b)$ Let $X$ be an i-Boolean algebra and let $x,y\in X$ such that $x\le y$, 
that is, $x\ra y=1$. 
Since $x\ra (x\ra y)^*=x\ra y^*$, we get $x^*=x\ra y^*$; hence, $x\le_L y$. \\
$(b)\Rightarrow (a)$ By Proposition \ref{ioml-50}$(7)$, $(x\wedge_Q y)\ra (y\wedge_Q x)=1$ and $(y\wedge_Q x)\ra (x\wedge_Q y)=1$, 
for all $x,y\in X$. It follows that $x\wedge_Q y\le y\wedge_Q x$ and $y\wedge_Q x\le x\wedge_Q y$, and by $(b)$, we get 
$x\wedge_Q y\le_L y\wedge_Q x$ and $y\wedge_Q x\le_L x\wedge_Q y$. Hence, $y\wedge_Q x=x\wedge_Q y$ for all $x,y\in X$, and applying 
Theorem \ref{cioml-100}, $X$ is an i-Boolean algebra. 
\end{proof}

\begin{corollary} \label{cioml-100-30} In any implicative-Boolean algebra, $\le = \le_L = \le_Q$.  
\end{corollary}

\begin{definition} \label{cioml-120} 
\emph{
The \emph{center} of an implicative-orthomodular lattice $X$ is the set 
$\mathcal{CC}(X)=\{x\in X\mid x \mathcal{C} y$ for all $y\in X\}$.  
}
\end{definition}

$\mathcal{CC}(X)$ is called the \emph{implicative-Boolean center} of $X$.

\begin{lemma} \label{cioml-120-10} Let $X$ be an implicative-orthomodular lattice. 
If $x,y,z\in X$ such that $x\mathcal{C} z$ and $y\mathcal{C} z$, then the following holds: \\
$\hspace*{2.00cm}$ $x\ra\ y\le_L ((x\ra y)\ra z^*)\ra ((x\ra y)\ra z)^*$. 
\end{lemma}
\begin{proof} Using Lemma \ref{iol-30-20}$(1)$,$(8)$,$(10)$, we have successively: \\
$\hspace*{2.00cm}$ $x^*, y\le_L x\ra y;$ \\
$\hspace*{2.00cm}$ $(x\ra y)\ra z^*\le_L x^*\ra z^*, y\ra z^*;$ \\
$\hspace*{2.00cm}$ $(x^*\ra z^*)^*, (y\ra z^*)^*\le_L ((x\ra y)\ra z^*)^*;$ \\
$\hspace*{2.00cm}$ $(x^*\ra z^*)\ra (y\ra z^*)^*\le_L ((x\ra y)\ra z^*)^*$. \\
Similarly we get $(x^*\ra z)\ra (y\ra z)^*\le_L ((x\ra y)\ra z)^*$. \\
Applying Lemma \ref{iol-30-20}$(12)$, it follows that: \\
$\hspace*{2.00cm}$ $((x^*\ra z^*)\ra (y\ra z^*)^*)^*\ra ((x^*\ra z)\ra (y\ra z)^*)$ \\
$\hspace*{6.00cm}$ $\le_L ((x\ra y)\ra z^*)\ra ((x\ra y)\ra z)^*$, \\ 
and by Proposition \ref{qbe-20}$(9)$, we get: \\
$\hspace*{2.00cm}$ $((x^*\ra z^*)\ra (x^*\ra z)^*)^*\ra ((y\ra z^*)\ra (y\ra z)^*)$ \\
$\hspace*{6.00cm}$ $\le_L ((x\ra y)\ra z^*)\ra ((x\ra y)\ra z)^*$. \\
Since $x\mathcal{C} z$ and $y\mathcal{C} z$, by Proposition \ref{cioml-60}, we have 
$(x^*\ra z^*)\ra (x^*\ra z)^*=x^*$ and $(y\ra z^*)\ra (y\ra z)^*=y$. 
Hence, $x\ra\ y\le_L ((x\ra y)\ra z^*)\ra ((x\ra y)\ra z)^*$. 
\end{proof}

Since implicative-Boolean subalgebras in an implicative-orthomodular lattice correspond to sets of commuting 
quantum propositions, it is important to investigate their existence. 

\begin{theorem} \label{cioml-130} 
The center $\mathcal{CC}(X)$ of an implicative-orthomodular lattice $X$ is an implicative-Boolean subalgebra 
of $X$.  
\end{theorem}
\begin{proof}
According to Lemma \ref{cioml-20}, $0,1\in \mathcal{CC}(X)$. Let $x,y\in \mathcal{CC}(X)$ and let $z\in X$, 
that is, $x\mathcal{C} z$ and $y\mathcal{C} z$. 
By Lemma \ref{cioml-120-10}, we have $x\ra y\le_L ((x\ra y)\ra z^*)\ra ((x\ra y)\ra z)^*$.   
On the other hand, by Lemma  \ref{iol-30-20}$(11)$, $((x\ra y)\ra z^*)\ra ((x\ra y)\ra z)^*\le_L x\ra y$; hence,   
$((x\ra y)\ra z^*)\ra ((x\ra y)\ra z)^*= x\ra y$. 
It follows that $(x\ra y)\mathcal{C} z$; therefore, $x\ra y\in \mathcal{CC}(X)$. 
Hence, $\mathcal{CC}(X)$ is a subalgebra of $X$. 
Since by Lemma \ref{cioml-90}, $x\mathcal{C} y$ if and only if $x\mathcal{D} y$, then $\mathcal{CC}(X)$ 
is an i-Boolean algebra. 
\end{proof}
\noindent

\begin{corollary} \label{cioml-130-10} An implicative-orthomodular lattice $X$ is an implicative-Boolean algebra  
if and only if $\mathcal{CC}(X)=X$, that is, $x\mathcal{C} y=y\mathcal{C} x$ for all $x,y\in X$, or equivalently, 
$x\mathcal{D} y=y\mathcal{D} x$ for all $x,y\in X$. 
\end{corollary}

\begin{theorem} \label{cioml-140} 
An implicative-ortholattice $X$ is an implicative-orthomodular lattice if and only if every two orthogonal 
elements of $X$ are contained in an implicative-Boolean subalgebra of $X$.
\end{theorem}
\begin{proof} 
Let $X$ be an i-OML and let $x,y\in X$ such that $x\perp y$. 
It follows that $y\perp x$; therefore, $x^*=x\ra y$ and $y^*=y\ra x$. 
Consider the subset $Y=\{0,x,y,x^*\ra y,x^*,y^*,(x^*\ra y)^*,1\}\subseteq X$. Using conditions $(impl)$ and $(pi)$, and Lemma \ref{ooml-30}, we have the following table: 
\[
\begin{array}{c|cccccccc}
\ra          & 0            & x        & y        & x^*\ra y & x^* & y^*          & (x^*\ra y)^* & 1 \\ \hline
0            & 1            & 1        & 1        & 1        & 1   & 1            & 1            & 1 \\
x            & x^*          & 1        & x^*      & 1        & x^* & 1            & x^*          & 1 \\
y            & y^*          & y^*      & 1        & 1        & 1   & y^*          & y^*          & 1 \\
x^*\ra y     & (x^*\ra y)^* & y^*      & x^*      & 1        & x^* & y^*          & (x^*\ra y)^* & 1 \\
x^*          & x            & x        & x^*\ra y & x^*\ra y & 1   & y^*          & y^*          & 1 \\
y^*          & y            & x^*\ra y & y        & x^*\ra y & x^* & 1            & x^*          & 1 \\
(x^*\ra y)^* & x^*\ra y     & x^*\ra y & x^*\ra y & x^*\ra y & 1   & 1            & 1            & 1 \\
1            & 0            & x        & y        & x^*\ra y & x^* & y^*          & (x^*\ra y)^* & 1 
\end{array}
.
\]

Using Lemma \ref{cioml-90-15}, one can easily verify that any two elements of $Y$ satisfy condition $(Idiv)$; 
hence, $(Y,\ra,^*,1)$ is an i-Boolean subalgebra of $X$ containing $x$ and $y$. 
To prove the converse, we use an idea from \cite{Paseka2}. 
Let $X$ be an i-OL and let $x,y\in X$ such that $x\le_L y$. 
It follows that $x\perp y^*$, and by hypothesis, $x$ and $y^*$ are contained in an i-Boolean subalgebra 
of $X$. Thus, the subalgebra $Y$ of $X$ generated by $x$ and $y$ is i-Boolean; hence, 
by Proposition \ref{cioml-90-20}, $Y$ is an i-OML. 
It follows that, by Theorem \ref{dioml-05-50}, $y\vee_Q x=y$. Since $x, y\in X$, by the same theorem, 
$X$ is an i-OML.
\end{proof}

\begin{example} \label{cioml-140-10} 
Consider the i-OL $(X,\ra,^*,1)$ from Example \ref{dioml-60}, with $X=\{0,a,b,c,d,1\}$. \\
$(1)$ Since $a^*=a\ra d$, we have $a\perp d$. 
If $Y\subseteq X$ is the subalgebra of $X$ generated by $a$ and $b$, then $Y$ also contains the element $d^*=b$.
Since $b\ra (b\ra a)^*=d\neq c=b\ra a^*$, condition $(Idiv)$ is not satisfied for the pair $(b, a)$. 
Hence, there exist orthogonal elements in $X$ which are not contained in any i-Boolean subalgebra of $X$.  
This result confirms that the hypothesis of Theorem \ref{cioml-140}, namely that $X$ must be an i-OML, is necessary. \\
$(2)$ It is easy to verify that $a\ra (a\ra b)^*=c=a\ra b^*$; hence, $a\mathcal{D} b$.   
In view of $(1)$, $a\mathcal{D} b$ does not imply $b\mathcal{D} a$; this confirms 
Corollary \ref{cioml-90-10}. \\
$(3)$ Since $b\ra a=1$ and $(b\ra a^*)^*=a\neq b$, we have $b\le a$ and $b\nleq_L a$, which is in accordance with 
Theorem \ref{cioml-100-20}. 
\end{example}

\begin{theorem} \label{cioml-150} 
$(\mathcal{SP}(\mathcal{CC}(X)),\circ,\varphi_1)$ is an Abelian monoid. 
\end{theorem}
\begin{proof} 
It is known that the composition of functions from a set into itself is associative. 
Since $a\mathcal{C} b$ for any $a,b\in \mathcal{CC}(X)$, then by Theorem \ref{cioml-60-40}, 
$\varphi_a\varphi_b=\varphi_b\varphi_a=\varphi_{a\wedge_Q b}$, that is, 
$\varphi_a\varphi_b\in \mathcal{SP}(\mathcal{CC}(X))$, and $\varphi_a$ commutes with $\varphi_b$. 
Moreover, $\varphi_1\in \mathcal{SP}(\mathcal{CC}(X))$, and $\varphi_a\varphi_1=\varphi_1\varphi_a=\varphi_a$, 
for any $a\in X$; thus, $\varphi_1$ is the identity element. 
Hence, $(\mathcal{SP}(\mathcal{CC}(X)),\circ,\varphi_1)$ is an Abelian monoid.  
\end{proof}

$\vspace*{1mm}$

\section{Sasaki set of projections on implicative-ortholattices} 

P.D. Finch proved in \cite{Finch} that an orthocomplemented poset is an orthomodular lattice if and only if it 
admits a special set of order-preserving maps, which he called projections. 
D.J. Foulis has shown that an orthomodular lattice $X$ can be represented by its Baer $^*$-semigroup, in which the 
projections are exactly the Sasaki projections on $X$
(a Baer $^*$-semigroup is an involution semigroup in which the right annihilator of each element is a principal right ideal generated by a projection. For details on this subject, the reader is referred to \cite{Foulis1, Foulis2, Eastdown}). \\
Similarly to \cite{Finch}, in the case of orthocomplemented posets, we introduce the Sasaki sets of projections  
and the full Sasaki set of projections on implicative-ortholattices, and we investigate their properties. 
As a main result, we prove that an implicative-ortholattice is an implicative-orthomodular lattice if
and only if it admits a full Sasaki set of projections.
This result will be used in the study of Sasaki spaces. 

\begin{definition} \label{ssioml-10} 
\emph{
Let $X$ be an implicative-ortholattice. A set $S$ of maps $X\longrightarrow X$ is said to be a 
\emph{Sasaki set of projections} on $X$ if it satisfies the following conditions: \\
$(SS_1)$ for each $\varphi\in S$, $x\le_L y$ implies $\varphi x\le_L \varphi y$ for all $x, y\in X;$ \\
$(SS_2)$ if $\varphi, \psi\in S$, $\varphi 1\le_L \psi 1$ implies $\varphi\psi=\varphi;$ \\
$(SS_3)$ for each $\varphi\in S$, $\varphi\varphi^* x\le_L x^*$ for all $x\in X$, \\
where $\varphi\psi=\varphi\circ \psi$ and $\varphi^*x=(\varphi x)^*$. 
}
\end{definition}

\begin{example} \label{ssioml-20} Let $X$ be an i-OL. 
Define $\varphi_0, \varphi_1:X\longrightarrow X$ by $\varphi_0 x=0$ and $\varphi_1 x=x$ for all $x\in X$. 
The set $S=\{\varphi_0, \varphi_1\}$ is a Sasaki set of projections on $X$. 
\end{example}

\begin{proposition} \label{ssioml-30} Let $X$ be an implicative-ortholattice and let $S$ be a Sasaki set 
of projections on $X$. The following hold for all $\varphi, \psi\in S$: \\
$(1)$ $\varphi 1\le_L \psi 1$ implies $\varphi\psi =\varphi=\psi\varphi;$ \\
$(2)$ $\varphi 1= \psi 1$ implies $\varphi=\psi;$ \\
$(3)$ $\varphi^2=\varphi$, where $\varphi^2=\varphi\varphi;$ \\
$(4)$ $\varphi 1\le_L \psi 1$ implies $\psi\varphi 1=\varphi 1;$ \\
$(5)$ $\varphi x=0$ iff $x\le_L \varphi^* 1$ for $x\in X;$ \\
$(6)$ $\varphi x\perp \varphi y$ implies $x\perp \varphi y$ for $x, y\in X;$ \\
$(7)$ $\varphi x\perp y$ iff $x\perp \varphi y$ for $x, y\in X$.   
\end{proposition}
\begin{proof}
$(1)$ By $(SS_3)$, $\psi(\psi\varphi x)^*\le_L (\varphi x)^*$ and 
$\varphi\psi(\psi\varphi x)^*\le_L \varphi(\varphi x)^*\le_L x^*$. 
From $(SS_2)$, $\varphi 1\le_L \psi 1$ implies $\varphi\psi=\varphi$; therefore,   
$\varphi(\psi\varphi x)^*=\varphi\psi(\psi\varphi x)^*\le_L x^*$.  
It follows that $x\le_L (\varphi(\psi\varphi x)^*)^*$ for all $x\in X$; hence, 
$\varphi x\le_L \varphi(\varphi(\psi\varphi x)^*)^*$. 
Taking $y=(\psi\varphi x)^*$, using $(SS_3)$ we have 
$\varphi x\le_L \varphi(\varphi(\psi\varphi x)^*)^*=\varphi(\varphi y)^*\le_L y^*=
((\psi\varphi x)^*)^*=\psi\varphi x$. Thus, $\varphi\le_L \psi\varphi$. \\
On the other hand, from $\varphi\psi=\varphi$, we get: 
$(\varphi x)^*=(\varphi\psi x)^*$, and applying $(SS_3)$, we have: 
$\psi\varphi(\varphi x)^*=\psi\varphi(\varphi\psi x)^*\le_L \psi(\psi x)^*
\le_L x^*$. 
Replacing $x$ by $(\varphi x)^*$, it follows that $\psi\varphi(\varphi(\varphi x)^*)^*\le_L \varphi x$.  
By $(SS_3)$, $\varphi(\varphi x)^*\le_L x^*$, hence, $x\le_L (\varphi(\varphi x)^*)^*$. 
Since $\psi\varphi$ is monotone, we get: 
$\psi\varphi x\le_L \psi\varphi(\varphi(\varphi x)^*)^*\le_L \varphi x$. 
Thus, $\psi\varphi\le_L \varphi $, and finally, $\psi\varphi=\varphi$. \\
$(2)$, $(3)$, $(4)$ follow from $(1)$. \\
$(5)$ If $\varphi x=0$, by $(SS_3)$, $\varphi(\varphi x)^*\le_L x^*$; hence, $\varphi(0^*)\le_L x^*$. 
It follows that $\varphi 1\le_L x^*$, that is, $x\le_L (\varphi 1)^*$. 
Conversely, if $x\le_L (\varphi 1)^*$, then by $(SS_1)$ and $(SS_3)$, we get 
$\varphi x\le_L \varphi(\varphi 1)^*\le_L 1^*=0$; hence, $\varphi x=0$. \\
$(6)$ If $\varphi x\perp \varphi y$, then $\varphi x\le_L \varphi^*y$, that is, $\varphi y\le_L \varphi^*x$. 
Using $(3)$ and $(SS_3)$, we obtain $\varphi\varphi y\le_L \varphi\varphi^*x\le_L x^*$; thus, $\varphi y\le_L x^*$. 
It follows that $\varphi y\perp x$, that is, $x\perp \varphi y$. \\
$(7)$ From $x\perp \varphi y$, we have $x\le_L \varphi^*y$. Hence, $\varphi x\le_L \varphi\varphi^*y\le_L y^*$, 
that is, $\varphi x\perp y$. 
Conversely, $\varphi x\perp y$ implies $\varphi x\le_L y^*$; thus, $y\le_L \varphi^*x$. 
It follows that $\varphi y\le_L \varphi\varphi^*x\le_L x^*$, that is, $\varphi y\perp x$. 
\end{proof}

\begin{proposition} \label{ssioml-60-20} 
Let $X$ be an implicative-ortholattice, let $S$ be a Sasaki set of projections on $X$, and 
let $\varphi\in S$. The following holds for all $x,y\in X$: \\
$\hspace*{5cm}$ $\varphi(x\ra y)=\varphi^* x^*\ra \varphi y$.  
\end{proposition}  
\begin{proof} 
Since $X$ is an i-OML, we have $\le_Q \iff \le_L$. 
Denote $z=\varphi^* x^*\ra \varphi y$. 
From $\varphi x^*\le_L \varphi^* x^*\ra \varphi y=z$, we get $z^*\le_L \varphi^* x^*$, and by $(SS_3)$, 
$\varphi z^*\le_L \varphi\varphi^* x^*\le_L x$. 
Hence, $x^*\le_L \varphi^* z^*$. 
Similarly, from $\varphi y\le_L \varphi^* x^*\ra \varphi y=z$, we have $y\le_L \varphi^* z^*$. 
By Lemma \ref{iol-30-20}$(10)$, $x\ra y\le_L \varphi^* z^*$, and by $(SS_3)$, 
$\varphi(x\ra y)\le_L \varphi\varphi^* z^* \le_L z=\varphi^* x^*\ra \varphi y$. 
Moreover, from $x^*,y\le_L x\ra y$, we have $\varphi x^*, \varphi y\le_L \varphi(x\ra y)$. 
Applying again Lemma \ref{iol-30-20}$(10)$, we get $\varphi^* x^*\ra \varphi y \le_L \varphi(x\ra y)$. 
We conclude that $\varphi(x\ra y)=\varphi^* x^*\ra \varphi y$. 
\end{proof}

Replacing $x$ by $x^*$, using the mutually transformations from Remark \ref{ioml-30-20}, we obtain 
$\varphi(x\vee y)=\varphi x\vee \varphi y$, that is, $\varphi$ preserves the join $\vee$. 

\begin{definition} \label{ssioml-40} 
\emph{
Let $X$ be an implicative-ortholattice. A Sasaki set of projections $S$ on $X$ is said to be \emph{full}, 
if for each $x\in X$, there exists $\varphi^x\in S$ such that $\varphi^x 1=x$. 
}
\end{definition}

\begin{remark} \label{ssioml-50} 
By Proposition \ref{ssioml-30}$(2)$, it follows that $\varphi^x$ from Definition \ref{ssioml-40} is unique. 
\end{remark}  

\begin{proposition} \label{ssioml-60} 
Let $X$ be an implicative-ortholattice admitting a full Sasaki set of projections $S$. 
Then, the following hold for all $x,y,z\in X$: \\
$(1)$ if $z\le_L x,y$, then $z\le \varphi^x(\varphi^x y^*)^*;$ \\
$(2)$ $\varphi^x(\varphi^x y^*)^*=(x\ra y^*)^*;$ \\
$(3)$ $\varphi^x x^*=0$. 
\end{proposition}  
\begin{proof} 
$(1)$ From $z\le_L x$, we have $\varphi^z 1\le_L \varphi^x 1$, and by Proposition \ref{ssioml-30}$(1)$, 
we get $\varphi^z\varphi^x=\varphi^z$. 
Since $y^*\le_L z^*=(\varphi^z 1)^*$, then by Proposition \ref{ssioml-30}$(5)$, we have $\varphi^z y^*=0$. 
It follows that  $\varphi^z\varphi^x(y^*)=\varphi^z(y^*)=0$, and applying again Proposition \ref{ssioml-30}$(5)$, 
we get $\varphi^x y^*\le_L (\varphi^z 1)^*=z^*$; hence, $z\le_L (\varphi^x y^*)^*$. 
From $\varphi^z 1=z\le_L x=\varphi^x 1$, by Proposition \ref{ssioml-30}$(4)$ we have 
$\varphi^x\varphi^z 1=\varphi^z 1$; thus, $\varphi^x z=z$. 
Finally, $z=\varphi^x z\le_L \varphi^x(\varphi^x y^*)^*$. \\
$(2)$ Since $(\varphi^x y^*)^*\le_L 1$, we have $\varphi^x(\varphi^x y^*)^*\le_L \varphi^x 1=x$. 
Moreover, by $(SS_3)$, $\varphi^x(\varphi^x y^*)^*\le_L (y^*)^*=y$. 
Applying Lemma \ref{qbe-20}$(8)$, we get $\varphi^x(\varphi^x y^*)^*\le_L (x\ra y^*)^*$. \\
On the other hand, from $(x\ra y^*)^*\le_L x, y$ (Lemma \ref{iol-30-20}$(4)$), using $(1)$, 
$(x\ra y^*)^*\le_L \varphi^x(\varphi^x y^*)^*$. 
We conclude that $\varphi^x(\varphi^x y^*)^*=(x\ra y^*)^*$. \\
$(3)$ Since $x^*\le_L x^*$ implies $x^*\le_L (\varphi^x 1)^*$, using Proposition \ref{ssioml-30}$(5)$, we get 
$\varphi^x x^*=0$. 
\end{proof}

The next result justifies using the term ``Sasaki set of projections" on an implicative-ortholattice $X$: 
for any $x\in X$, $\varphi^x$ is a Sasaki projection. 

\begin{proposition} \label{ssioml-60-10} 
Let $X$ be an implicative-ortholattice that admits a full Sasaki set of projections $S$. 
Then $\varphi^x y= y\wedge_Q x$ for all $x,y\in X$. 
\end{proposition}  
\begin{proof} 
For all $x,y\in X$, we have: \\
$\hspace*{3.00cm}$ $y\wedge_Q x=((y^*\ra x^*)\ra x^*)^*=(x\ra (y^*\ra x^*)^*)^*$ \\ 
$\hspace*{4.00cm}$ $=\varphi^x(\varphi^x(y^*\ra x^*)^*)^*$ (Prop. \ref{ssioml-60}$(2))$ \\ 
$\hspace*{4.00cm}$ $=\varphi^x(\varphi^x(x\ra y)^*)^*$ \\                             
$\hspace*{4.00cm}$ $=\varphi^x(\varphi^x\varphi^x (\varphi^x y)^*)^*$ (Prop. \ref{ssioml-60}$(2))$ \\
$\hspace*{4.00cm}$ $=\varphi^x(\varphi^x(\varphi^x y)^*)^*$ (Prop. \ref{ssioml-30}$(3))$ \\
$\hspace*{4.00cm}$ $=\varphi^x(x\ra y)$ (Prop. \ref{ssioml-60}$(2))$ \\
$\hspace*{4.00cm}$ $=\varphi^x(y^*\ra x^*)$ \\
$\hspace*{4.00cm}$ $=(\varphi^x y)^*\ra \varphi^x(x^*)$ (Prop. \ref{ssioml-60-20}) \\
$\hspace*{4.00cm}$ $=(\varphi^x y)^*\ra 0$ (Prop. \ref{ssioml-60}$(3))$ \\
$\hspace*{4.00cm}$ $=\varphi^x y$. 
\end{proof}

Based on Sasaki set of projections, we give a characterization theorem for implicative-orthomodular lattices. 
This result will be useful in the study of Sasaki orthogonality spaces. 

\begin{theorem} \label{ssioml-70} An implicative-ortholattice is an implicative-orthomodular lattice if and 
only if it admits a full Sasaki set of projections. 
\end{theorem}  
\begin{proof}
Let $(X,\ra,^*,1)$ be an i-OL and let $S$ be a full Sasaki set of projections on $X$. 
Consider $x,y\in X$ such that $x\le y$ and $y\le_L x$. Then there exists $\varphi^x, \varphi^y\in S$ such that $\varphi^x 1=x$ and $\varphi^y 1=y$. 
Applying Proposition \ref{ssioml-60}$(2)$, it follows that $\varphi^x(\varphi^x y)^*=(x\ra y)^*$. 
From $y\le_L x$, we have $\varphi^y 1\le_L \varphi^x 1$, and by Proposition \ref{ssioml-30}$(1)$, 
$\varphi^x\varphi^y 1=\varphi^y 1$, that is, $\varphi^x y=y$. 
Then $(\varphi^x y)^*=y^*$, and $\varphi^x(\varphi^x y)^*=\varphi^x y^*$; 
thus, $\varphi^x y^*=(x\ra y)^*=0$, since $x\le y$. 
By Proposition \ref{ssioml-30}$(5)$, we get $y^*\le_L (\varphi^x 1)^*=x^*$, that is, $x\le_L y$. 
Thus, $x=y$, and according to Theorem \ref{dioml-05-60}, $X$ is an i-OML. \\
Conversely, if $X$ is an i-OML, let $S=\mathcal{SP}(X)=\{\varphi_a\mid a\in X\}$, the set of all Sasaki 
projections on $X$. 
By Propositions \ref{sioml-20}$(5)$, \ref{sioml-30}$(4)$, it follows that $(SS_1)$ and $(SS_3)$ hold for any $\varphi_a\in S$. 
If $\varphi_a, \varphi_b \in S$ such that $\varphi_a 1\le_L \varphi_b 1$, then $a\le_L b$, and applying 
Proposition \ref{ioml-50-10}$(2)$, we have $(x\wedge_Q b)\wedge_Q a=x\wedge_Q a$ for all $x\in X$. 
Hence, $\varphi_a\varphi_b=\varphi_a$, and condition $(SS_2)$ is verified. 
It follows that $S$ is a Sasaki set of projections on $X$. 
For any $x\in X$, taking $\varphi^x:=\varphi_x\in S$, we have $\varphi^x 1=\varphi_x 1=x$; hence, $S$ is a full 
Sasaki set of projections on $X$. 
\end{proof}

\begin{example} \label{ssioml-80}
Let $(X,\ra,^*,1)$ be the i-OL from Example \ref{dioml-60}, where $X=\{0, a, b, c, d, 1\}$. 
We show that there is no full set of projections on $X$. 
Suppose, for contradiction, that $X$ admits a full set of projections $S$. 
It follows that there exists $\varphi^a, \varphi^b\in S$ such that $\varphi^a1=a$ and $\varphi^b1=b$. 
Since $(a\ra b^*)^*=(a\ra d)^*=c^*=a$, we have $a\le_L b$. 
Hence, $\varphi^a1\le_L \varphi^b1$, and, by Proposition \ref{ssioml-30}$(1)$, we obtain 
$\varphi^b\varphi^a1=\varphi^a1$, that is, $\varphi^ba=a$. 
Applying Proposition \ref{ssioml-60}$(2)$, we have $\varphi^b(\varphi^ba)^*=(b\ra a)^*=1^*=0$. 
Thus, $\varphi^ba^*=0$, and, by Proposition \ref{ssioml-30}$(5)$, $a^*\le_L (\varphi^b1)^*$, that is, $a^*\le_L b^*$.  
It follows that $b\le_L a$, and together with $a\le_L b$, we obtain $a=b$, a contradiction. 
Hence, there is no full set of projections on $X$. 
\end{example}

\begin{example} \label{ssioml-90}
Consider the set $X=\{0,a,b,c,d,1\}$, and the operations $\ra$, $\wedge_Q$ defined below. 
\[
\begin{picture}(50,-70)(0,40)
\put(37,11){\line(-5,4){32}} 
\put(18,37){\circle*{3}}
\put(-5,35){$a$}

\put(37,11){\line(-3,4){20}} 
\put(6,37){\circle*{3}}
\put(23,35){$b$}

\put(37,11){\line(2,5){10}} 
\put(47,37){\circle*{3}}
\put(38,35){$c$}

\put(39,11){\line(3,4){20}}
\put(37,11){\circle*{3}}
\put(34,0){$0$}

\put(18,37){\line(3,4){17}}
\put(34,59){\circle*{3}}
\put(32,64){$1$}

\put(47,37){\line(-3,5){12}}  

\put(6,37){\line(5,4){28}}

\put(60,37){\line(-6,5){25}}  
\put(59,37){\circle*{3}}
\put(65,35){$d$}
\end{picture}
\hspace*{2cm}
\begin{array}{c|cccccc}
\rightarrow & 0 & a & b & c & d & 1 \\ \hline
0 & 1 & 1 & 1 & 1 & 1 & 1 \\
a & b & 1 & b & 1 & 1 & 1 \\
b & a & a & 1 & 1 & 1 & 1 \\
c & d & 1 & 1 & 1 & d & 1 \\
d & c & 1 & 1 & c & 1 & 1 \\
1 & 0 & a & b & c & d & 1
\end{array}
\hspace{10mm}
\begin{array}{c|ccccccc}
\wedge_Q & 0 & a & b & c & d & 1 \\ \hline
0    & 0 & 0 & 0 & 0 & 0 & 0 \\ 
a    & 0 & a & 0 & c & d & a \\ 
b    & 0 & 0 & b & c & d & b \\ 
c    & 0 & a & b & c & 0 & c \\
d    & 0 & a & b & 0 & d & d \\
1    & 0 & a & b & c & d & 1 
\end{array}.
\]

Then $(X,\ra,^*,1)$, where $x^*=x\ra 0$ is an i-OML. 
Let the maps $\varphi_0, \varphi_a, \varphi_b, \varphi_c, \varphi_d, \varphi_1:X\longrightarrow X$ be defined 
as follows. 

\[
\begin{array}{c|ccccccc}
x          & 0 & a & b & c & d & 1 \\ \hline
\varphi_0x=x\wedge_Q 0 & 0 & 0 & 0 & 0 & 0 & 0 \\ 
\varphi_ax=x\wedge_Q a & 0 & a & 0 & a & a & a \\ 
\varphi_bx=x\wedge_Q b & 0 & 0 & b & b & b & b \\ 
\varphi_cx=x\wedge_Q c & 0 & c & c & c & 0 & c \\
\varphi_dx=x\wedge_Q d & 0 & d & d & 0 & d & d \\
\varphi_1x=x\wedge_Q 1 & 0 & a & b & c & d & 1 
\end{array}.
\]

From the proof of Theorem \ref{ssioml-70}, 
$S=\mathcal{SP}(X)=\{\varphi_0, \varphi_a, \varphi_b, \varphi_c, \varphi_d, \varphi_1\}$ is a full set of 
projections on $X$. 
\end{example}

$\vspace*{1mm}$

\section{Connection with Dacey and Sasaki spaces} 

The logic of an orthogonality space is an ortholattice (namely, the ortholattice of its orthoclosed subsets), 
but it is not necessarily orthomodular. 
J.R. Dacey introduced orthogonality spaces in which every orthoclosed set is generated, via double orthogonal closure, by any maximal orthogonal subset it contains \cite{Dacey}. This condition guarantees that the logic of an 
orthogonality space is orthomodular, and the orthogonality spaces satisfying Dacey's condition are called Dacey spaces. 
In this section, we investigate these notions in the case of implicative-ortholattices, and we define and characterize the Dacey and Sasaki spaces. 

Given an implicative-ortholattice $(X,\ra,^*,1)$, we defined in Section 3 its associated orthogonality space 
$(X^{\prime},\perp)$, where $X^{\prime}=X\setminus \{0\}$ and $x\perp y$ if and only if $x^*=x\ra y$. 

\begin{definition} \label{dacey-30} \emph{
Let $(X,\ra,^*,1)$ be an implicative-ortholattice. The orthogonality space $(X^{\prime},\perp)$ is a 
\emph{Dacey space} if $\mathcal{CL}(X^{\prime},\perp)$ is an implicative-orthomodular lattice. 
}
\end{definition}

\begin{proposition} \label{dacey-40} Let $(X,\ra,^*,1)$ be an implicative-ortholattice and let 
$(X^{\prime},\perp)$ be the associated orthogonality space. The following are equivalent: \\ 
$(a)$ $(X^{\prime},\perp)$ is a Dacey space; \\
$(b)$ Every two orthogonal elements of $\mathcal{CL}(X^{\prime},\perp)$ are contained in an implicative-Boolean subalgebra of $\mathcal{CL}(X^{\prime},\perp)$.  
\end{proposition}
\begin{proof}
It follows from Theorem \ref{cioml-140}. 
\end{proof}

The set $P(H)$ of one-dimensional subspaces of a Hilbert space $H$ is an orthogonality space under  
the orthogonality relation $\perp$ induced by the inner product on $H$. 
B. Lindenhovius and T. Vetterlein introduced in \cite{Lind1} the Sasaki maps in order to establish conditions 
under which an orthogonality space $(X,\perp)$ is isomorphic to $(P(H),\perp)$ for some orthomodular space $H$. \\
We adapt these notions to the case of implicative-ortholattices, proving that the associated orthogonality 
space of a complete implicative-orthomodular lattice is both a Sasaki space and a Dacey space. 
We also show that any complete implicative-ortholattice with a full set of projections is a Sasaki space. \\

For a subset $Y\subseteq X$, we denote by $\complement Y$ the complement of $Y$ with respect to $X$. 
We will also use the notation $\complement Y^{\perp}$ instead of $\complement (Y^{\perp})$. 

\begin{definition} \label{sasm-10} 
\emph{
Let $(X^{\prime},\perp)$ be the associated orthogonality space of an implicative-ortholattice, and 
let $A$ be an orthoclosed subset of $X^{\prime}$. 
A \emph{Sasaki map} to $A$ is a map $\varphi: \complement A^{\perp}\longrightarrow A$ satisfying the following conditions: \\
$(SM_1)$ $\varphi x=x$ for all $x\in A;$ \\
$(SM_2)$ for any $x,y\in \complement A^{\perp}$, $\varphi x\perp y$ if and only if $x\perp \varphi y$. \\
We call $(X^{\prime},\perp)$ a \emph{Sasaki space} if for any orthoclosed subset $A$ of $X^{\prime}$, there 
exists a Sasaki map $\varphi$ to $A$.
}
\end{definition}  

\noindent
According to \cite{Lind1}, any Sasaki space is a Dacey space. \\
We say that an implicative-ortholattice (implicative-orthomodular lattice) $(X,\ra,^*,1)$ is \emph{complete} 
if the corresponding lattice $(X,\wedge^P,\vee^P,0,1)$ (see Remark \ref{qbe-10-10}) is complete. \\
For any $x\in X$, denote $\downarrow x=\{y\in X\mid y\le_L x\}$. 
If $Y\subseteq X$, denote $\downarrow Y=\bigcup_{y\in Y}(\downarrow y)$. 

\begin{lemma} \label{sasm-10-10} Let $X$ be an implicative-ortholattice. 
The following hold for all $x,y\in X$: \\
$(1)$ $(\downarrow x)^{\perp}\cap X^{\prime}=\downarrow x^*\cap X^{\prime}=x^{\perp}\cap X^{\prime};$ \\
$(2)$ $(\downarrow x)\cap (\downarrow y)=\downarrow (x\ra y^*)^*;$ \\
$(3)$ $\downarrow(x\ra y)=(\downarrow x)\ra (\downarrow y)$. \\
If $X$ is complete and $Y\subseteq X$ is non-empty, then: \\
$(4)$ $\bigwedge_{y\in Y}(\downarrow y)=\downarrow (\bigwedge_{y\in Y}y);$ \\
$(5)$ $Y^{\perp}=\downarrow (\bigwedge_{y\in Y}y^*)$. 
\end{lemma}
\begin{proof}
$(1)$ First we show that $\downarrow x \cap \downarrow x^*=\{0\}$ and $\downarrow x\cap x^{\perp}=\{0\}$. 
Indeed, $u\in \downarrow x\cap \downarrow x^*$ implies $u\le_L x$ and $u\le_L x^*$, that is, 
$u\le_L x$ and $x\le_L u^*$. It follows that $u\le_L u^*$; hence, $u^*=u\ra u=1$ and $u=0$. 
Similarly, for any $u\in\downarrow x\cap x^{\perp}$ we have $u\le_L x$ and $u\perp x$; 
thus, $u\le_L x$ and $u\le_L x^*$. Finally, we get $u\le_L x\le_L u^*$, hence, $u=0$. \\
For any $u\in (\downarrow x)^{\perp}$, we have $u\perp x$ iff $u^*=u\ra x$ iff $u\le_L x^*$ iff 
$u\in \downarrow x^*$; hence, $(\downarrow x)^{\perp}\subseteq \downarrow x^*$. 
Conversely, if $u\in \downarrow x^*$, then $u\le_L x^*$. For all $v\in X^{\prime}$, $v\le_L x$, we have 
$x^*\le_L v^*$. It follows that $u\le_L v^*$; thus, $u^*=u\ra v$, that is, $u\perp v$. 
Hence, $u\in (\downarrow x)^{\perp}$; thus, $\downarrow x^*\subseteq (\downarrow x)^{\perp}$. 
It follows that $(\downarrow x)^{\perp}\cap X^{\prime}=\downarrow x^*\cap X^{\prime}$. 
Moreover, $u\in \downarrow x^*$ iff $u\le_L x^*$ iff $u^*=u\ra x$ iff $u\perp x$ iff $x\in x^{\perp}$. 
It follows that $\downarrow x^*\cap X^{\prime}=x^{\perp}\cap X^{\prime}$. \\
$(2)$ If $u\in (\downarrow x)\cap (\downarrow y)$, we have $u\le_L x, y$; hence, $u\le_L(x\ra y^*)^*$ 
(by Prop. \ref{qbe-20}$(8)$).  
It follows that $u\in \downarrow (x\ra y^*)^*$; thus, 
$(\downarrow x)\cap (\downarrow y)\subseteq \downarrow (x\ra y^*)^*$.
Conversely, $u\in \downarrow (x\ra y^*)^*$ implies $u\le_L (x\ra y^*)^*\le_L x,y$, that is, 
$u\in (\downarrow x)\cap (\downarrow y)$; hence, 
$\downarrow (x\ra y^*)^*\subseteq (\downarrow x)\cap (\downarrow y)$. \\
$(3)$ Using $(2)$, we have: \\
$\hspace*{2.00cm}$ $\downarrow(x\ra y)=\downarrow(x\ra y)^{**}=(\downarrow(x\ra y)^*)^{\perp}
                   =(\downarrow x\cap \downarrow y^*)^{\perp}$ \\
$\hspace*{3.70cm}$ $=(\downarrow x\cap (\downarrow y)^{\perp})^{\perp}
                    =(\downarrow x)\ra (\downarrow y)$. \\
$(4)$ For any $u\in X^{\prime}$, we have: 
$u\in \bigwedge_{y\in Y}(\downarrow y)$ iff $u\in \downarrow y$ for all $y\in Y$ iff $u \le_L y$ for all $y\in Y$ iff 
   $u\le_L \bigwedge_{y\in Y}y$ for all $y\in Y$ iff $u\in \downarrow (\bigwedge_{y\in Y}y)$. 
Hence, $\bigwedge_{y\in Y}(\downarrow y)=\downarrow (\bigwedge_{y\in Y}y)$. \\
$(5)$ Using $(3)$, we get: \\
$\hspace*{2.00cm}$ $Y^{\perp}=\{x\in X\mid x\perp y$ for all $y\in Y\}
                             =\{x\in X\mid x\le_L y^*$ for all $y\in Y\}$ \\
$\hspace*{2.70cm}$ $=\{x\in X\mid x\in \downarrow y^*$ for all $y\in Y\}                             
                             = \bigwedge_{y\in Y}(\downarrow y^*)=\downarrow(\bigwedge_{y\in Y}y^*)$. 
\end{proof}

Following the idea from \cite[Prop. 2.8]{Lind1}, we show that any Sasaki projection on a complete implicative-orthomodular lattice induces a Sasaki map on the associated orthogonality space. 

\begin{proposition} \label{sasm-20-10} Let $(X,\ra,^*,1)$ be a complete implicative-ortholattice. 
The following hold: \\
$(1)$ $\mathcal{CL}(X^{\prime},\perp)=\{\downarrow x\cap X^{\prime}\mid x\in X\};$ \\
$(2)$ $X$ and $\mathcal{CL}(X^{\prime},\perp)$ are isomorphic as implicative-ortholattices. 
\end{proposition} 
\begin{proof}
$(1)$ The proof follows similar to that in \cite[Prop. 2.8]{Lind1}. 
By Lemma \ref{sasm-10-10}, for any non-empty subset $Y\subseteq X$ we have 
$Y^{\perp}=\downarrow (\bigwedge_{y\in Y}y^*)$. 
Since for any $A\in \mathcal{CL}(X^{\prime},\perp)$, $A=Y^{\perp}$ for some $Y\subseteq X$, it follows that 
$\mathcal{CL}(X^{\prime},\perp)=\{\downarrow x\cap X^{\prime}\mid x\in X\}$. \\
$(2)$ Define $H:X\longrightarrow \mathcal{CL}(X^{\prime},\perp)$ by $H(x)=\downarrow x\cap X^{\prime}$. \\
Obviously, $x\ne y$ implies $\downarrow x\cap X^{\prime} \ne \downarrow y\cap X^{\prime}$. 
Since $\mathcal{CL}(X^{\prime},\perp)=\{\downarrow x\cap X^{\prime}\mid x\in X\}$, for any 
$A\in \mathcal{CL}(X^{\prime},\perp)$ there exists $x\in X$ such that $H(x)=A$.
Applying Lemma \ref{sasm-10-10}, we get: \\
$\hspace*{2.00cm}$ $H(x\ra y)=\downarrow(x\ra y)\cap X^{\prime}=(\downarrow x\ra \downarrow y)\cap X^{\prime}
                             =H(x)\ra H(y)$, \\
$\hspace*{2.00cm}$ $H(x^*)=\downarrow x^*\cap X^{\prime}=(\downarrow x)^{\perp}\cap X^{\prime}=(H(x))^{\perp}$. \\ 
It folllows that $H$ is an isomorphism of i-OLs. 
\end{proof}

\begin{theorem} \label{sasm-20} Let $(X,\ra,^*,1)$ be a complete implicative-orthomodular lattice. 
Then the associated orthogonality space $(X^{\prime},\perp)$ is a Sasaki space. 
\end{theorem}
\begin{proof} 
Let $A\in \mathcal{CL}(X^{\prime},\perp)=\{\downarrow x\cap X^{\prime}\mid x\in X\}$, 
with $A=\downarrow x\cap X^{\prime}$ for some $x\in X$. 
Consider the map $\varphi_A:\complement A^{\perp}\longrightarrow A$ defined by 
$\varphi_A x=\varphi{_x}_{\mid \complement A^{\perp}}$, where $\varphi_x:X\longrightarrow X$ is the Sasaki projection on $X$ defined by $x$, and $\varphi{_x}_{\mid \complement A^{\perp}}$ is the restriction of $\varphi{_x}$ to 
$\complement A^{\perp}$. 
We have: \\ 
$\hspace*{2.00cm}$ $\complement A^{\perp}=\complement(\downarrow x\cap X^{\prime})^{\perp}
                   =\complement(H(x))^{\perp}
                   =\complement H(x^*)=X^{\prime}\setminus(\downarrow x^*\cap X^{\prime})
                    =X^{\prime}\setminus \downarrow x^*$,  \\
where $H$ is the isomorphism defined in Proposition \ref{sasm-20-10}.                     
Hence, for any $y\in \complement A^{\perp}$, $y\notin \downarrow x^*$; it follows that $y\nleq x^*$. 
Applying Proposition \ref{sioml-40}$(2)$, $\varphi_x y\ne 0$. 
Moreover, since $\varphi_A y=\varphi_x y=y\wedge_Q x\le_L x$, we have $\varphi_A x\in A$. 
Hence, $\varphi_A$ is well defined. 
For any $y\in A$, we have $y\in \downarrow x$, that is, $y\le_L x$. Since, by Proposition \ref{sioml-40}$(1)$, 
we have $\varphi_x y=y$, it follows that condition $(SM_1)$ is verified. 
Since for all $x,y\in \complement A^{\perp}$, by Proposition \ref{sioml-40}$(6)$ we have $\varphi_A x\perp y$ 
iff $x\perp \varphi_A y$, condition $(SM_2)$ is also satisfied. 
Hence, $\varphi_A$ is a Sasaki map, that is, $(X^{\prime},\perp)$ is a Sasaki space. 
\end{proof}

\begin{proposition} \label{sasm-50} If a complete implicative-ortholattice admits a full Sasaki set of projections, then its associated orthogonality space is a Sasaki space. 
\end{proposition} 
\begin{proof}
Let $(X,\ra,^*,1)$ be a complete i-OL and let $(X^{\prime},\perp)$ be the associated orthogonality space ($X^{\prime}=X\setminus \{0\}$). 
Since $X$ admits a full Sasaki set of projections, then by Theorem \ref{ssioml-70} it is a complete  implicative-orthomodular lattice. Applying Theorem \ref{sasm-20}, for any orthoclosed subset $A$ of $X^{\prime}$ there exists a Sasaki map to $A$. Hence, $(X^{\prime},\perp)$ is a Sasaki space.  
\end{proof}

A \emph{decomposition} of an orthogonality space $(X,\perp)$ is a pair $(A_1,A_2)$ of non-empty subsets of $X$ 
such that $A_1=A_2^{\perp}$ and $A_2=A_1^{\perp}$. Clearly, $A_1\perp A_2$ and $(A_1\cup A_2)^{\perp}=\emptyset$. 
A maximal $\perp$-set of $X$ is called a \emph{block}, and a \emph{block partition} is a partition of a block 
of $X$ into non-empty subsets called \emph{cells} \cite{Paseka2}. \\
Motivated by the example of $(P(H),\perp)$, J. Paseka and T. Vetterlein investigated in \cite{Paseka2} the conditions under which any block partition of an orthogonality space gives rise to a unique decomposition of the space.
They called such spaces normal orthogonality spaces. \\
We adapt the notion of normal orthogonality spaces to the case of implicative-ortholattices. 
Given an implicative-ortholattice $X$, if $(X^{\prime},\perp)$ is normal, we prove that the set of all its orthoclosed subsets forms an implicative-Boolean subalgebra of $\mathcal{CL}(X^{\prime},\perp)$. 

\begin{definition} \label{normal-20} \cite{Paseka2} 
\emph{
An orthogonality space $(X,\perp)$ is called \emph{normal} if for any block partition $(E_1,E_2)$ there is a unique 
decomposition $(\overline{E_1}, \overline{E_2})$ such that $E_1\subseteq \overline{E_1}$, 
$E_2\subseteq \overline{E_2}$. 
}
\end{definition}

\begin{lemma} \label{normal-30} \cite[Prop. 2.5]{Paseka2} Let $(X,\perp)$ be an orthogonality space. 
Then $X$ is normal if and only if for any block partition $(E_1,E_2)$, the pair $(E_2^{\perp},E_1^{\perp})$ is a decomposition of $X$. 
\end{lemma}

The next result is inspired by \cite[Prop. 2.8]{Paseka2}. 

\begin{proposition} \label{normal-40} Let $(X,\ra,^*,1)$ be an implicative-ortholattice such that  
the associated orthogonality space $(X^{\prime},\perp)$ is normal. For any maximal $\perp$-subset $E$ of $X$, 
denote $\mathcal{B}_E=\{A^{\perp\perp}\mid A\subseteq E\}$. 
Then $(\mathcal{B}_E, \ra, ^{\perp}, X)$ is an implicative-Boolean subalgebra of $\mathcal{CL}(X^{\prime},\perp)$.  
\end{proposition}
\begin{proof} 
Clearly, $\mathcal{B}_E \subseteq \mathcal{CL}(X^{\prime},\perp)$ and $\emptyset \in \mathcal{B}_E$. 
Obviously, $E^{\perp\perp}\subseteq X$. Assume that there exists $x\in X$, $x\notin E^{\perp\perp}$. 
Then $E^{\perp}=\emptyset$ (indeed, if $e\in E^{\perp}$, then $e\perp E$; thus, $E\cup \{e\}$ is a $\perp$-set 
contained in $X$; hence, $E$ is not maximal, a contradiction). It follows that $E^{\perp\perp}=X$, with 
$E\subseteq E$, that is, $X \in \mathcal{B}_E$. 
One can easy prove that $\mathcal{B}_E$ is closed under arbitrary meets and joins (see also \cite[Prop. 2.8]{Paseka2}). 
Let $A^{\perp\perp}\in \mathcal{B}_E$. Since $(A,E\setminus A)$ is a block partition, then  
by Lemma \ref{normal-30}, $((E\setminus A)^{\perp},A^{\perp})$ is a decomposition of $X$. 
Hence, $A^{\perp}\perp (E\setminus A)^{\perp}$. It follows that $A^{\perp}= (E\setminus A)^{\perp\perp}$ and  
$(A^{\perp\perp})^{\perp}=A^{\perp}=(E\setminus A)^{\perp\perp}\in \mathcal{B}_E$; 
hence, $\mathcal{B}_E$ is closed under $^{\perp}$.
For any $A^{\perp\perp},B^{\perp\perp}\in \mathcal{B}_E$, we have 
$(A^{\perp\perp} \ra B^{\perp\perp})^{\perp\perp}=(A^{\perp\perp}\cap B^{\perp})^{\perp}\in \mathcal{B}_E$. 
Moreover, 
$A^{\perp\perp}\ra (A^{\perp\perp}\ra B^{\perp\perp})^{\perp}=A^{\perp}\vee (A^{\perp}\vee B^{\perp\perp})^{\perp}
=A^{\perp}\vee (A^{\perp\perp}\wedge B^{\perp})=A^\perp\vee B^\perp=(A^{\perp\perp}\wedge B^{\perp\perp})^{\perp}
=A^{\perp\perp}\ra (B^{\perp\perp})^{\perp}$, that is, $\mathcal{B}_E$ satisfies condition $\mathcal{D}$. 
Hence, $\mathcal{B}_E$ is an i-Boolean subalgebra of $\mathcal{CL}(X^{\prime},\perp)$.  
\end{proof}

\begin{example} \label{sasm-30}
Consider the i-OL $(X,\ra,^*,1)$ from Example \ref{ooml-60}, with $X=\{0,a,b,c,d,1\}$. 
The associated orthogonality space of $X$ is $(X^{\prime},\perp)$, where $X^{\prime}=X\setminus \{0\}$. 
By letting $A=\{a\}$, $B=\{d\}$, $C=\{a,b\}$, $D=\{c,d\}$, we have shown that $\mathcal{CL}(X^{\prime},\perp)=\{\emptyset, A, B, C, D, X^{\prime}\}$ and 
$(\mathcal{CL}(X^{\prime},\perp),\ra,^*,X^{\prime})$ is an i-OL, but not an i-OML. 
It follows that $(X^{\prime},\perp)$ is not a Dacey space. 
We have $C^{\perp}=\{d\}$ and $\complement C^{\perp}=\{a,b,c,1\}$. 
Assume that there exists a Sasaki map \\
$\varphi: \complement C^{\perp}\longrightarrow C$; hence, by $(SS_1)$, 
$\varphi a=a$ and $\varphi b=b$. For $a,c\in \complement C^{\perp}$ we have $\varphi a=a\perp c$, 
but $a\not\perp \varphi c\in \{a,b\}$. Hence, condition $(SS_2)$ is not satisfied, so there are no 
Sasaki maps on $C=\{a,b\}$. Thus, $(X^{\prime},\perp)$ is not a Sasaki space. 
\end{example}

\begin{example} \label{sasm-40} 
Consider the set $X=\{0,a,b,c,d,1\}$ and the operations $\ra$, $\wedge_Q$ defined below.
\[
\begin{picture}(50,-70)(0,40)
\put(37,11){\line(-5,4){32}} 
\put(18,37){\circle*{3}}
\put(-5,35){$a$}

\put(37,11){\line(-3,4){20}} 
\put(6,37){\circle*{3}}
\put(23,35){$b$}

\put(37,11){\line(2,5){10}} 
\put(47,37){\circle*{3}}
\put(38,35){$c$}

\put(39,11){\line(3,4){20}}
\put(37,11){\circle*{3}}
\put(34,0){$0$}

\put(18,37){\line(3,4){17}}
\put(34,59){\circle*{3}}
\put(32,64){$1$}

\put(47,37){\line(-3,5){12}}  

\put(6,37){\line(5,4){28}}

\put(60,37){\line(-6,5){25}}  
\put(59,37){\circle*{3}}
\put(65,35){$d$}
\end{picture}
\hspace*{2cm}
\begin{array}{c|cccccc}
\rightarrow & 0 & a & b & c & d & 1 \\ \hline
0 & 1 & 1 & 1 & 1 & 1 & 1 \\
a & d & 1 & 1 & 1 & d & 1 \\
b & c & 1 & 1 & c & 1 & 1 \\
c & b & 1 & b & 1 & 1 & 1 \\
d & a & a & 1 & 1 & 1 & 1 \\
1 & 0 & a & b & c & d & 1
\end{array}
\hspace{10mm}
\begin{array}{c|ccccccc}
\wedge_Q & 0 & a & b & c & d & 1 \\ \hline
0        & 0 & 0 & 0 & 0 & 0 & 0 \\ 
a        & 0 & a & b & c & 0 & a \\ 
b        & 0 & a & b & 0 & d & b \\ 
c        & 0 & a & 0 & c & d & c \\
d        & 0 & 0 & b & c & d & d \\
1        & 0 & a & b & c & d & 1 
\end{array}.
\]


Then $(X,\ra,^*,1)$, where $x^*=x\ra 0$ is an i-OML. 
Denoting $X^{\prime}=X\setminus \{0\}$, then $(X^{\prime},\perp)$ is an orthogonality space, with $\perp$ defined 
in Section 3.  
We can see that $a^{\perp}=\{d\}$, $b^{\perp}=\{c\}$, $c^{\perp}=\{b\}$, $d^{\perp}=\{a\}$, 
$1^{\perp}=\emptyset$ and $\mathcal{CL}(X^{\prime},\perp)=\{\emptyset, \{a\}, \{b\}, \{c\}, \{d\}, X^{\prime}\}$. 
The maps: \\
$\hspace*{4cm}$ $\varphi^a:\complement a^{\perp}\longrightarrow \{a\}$, $\varphi^a x=a$, \\
$\hspace*{4cm}$ $\varphi^b:\complement b^{\perp}\longrightarrow \{b\}$, $\varphi^b x=b$, \\
$\hspace*{4cm}$ $\varphi^c:\complement c^{\perp}\longrightarrow \{c\}$, $\varphi^c x=c$, \\
$\hspace*{4cm}$ $\varphi^d:\complement d^{\perp}\longrightarrow \{d\}$, $\varphi^d x=d$, \\
are Sasaki maps; hence, $(X^{\prime},\perp)$ is a Sasaki space. 
It follows that $(X^{\prime},\perp)$ is also a Dacey space.
\end{example} 

\begin{example} \label{nos-20} Consider the orthogonality space $(X^{\prime},\perp)$ from Example \ref{sasm-30}, 
where $\{a,d\}$, $\{a,c\}$, $\{b,d\}$ are the blocks of $X^{\prime}$. 
Denoting $E_1=\{a\}$, $E_2=\{c\}$, $\overline{E_1}=\{a\}$, $\overline{E_2}=\{c,d\}$, we have 
$E_1\subseteq \overline{E_1}$, $E_2\subseteq \overline{E_2}$, 
$\overline{E_1}^{\perp}=a^{\perp}=\{c,d\}=\overline{E_2}$ and $\overline{E_2}^{\perp}=\{c,d\}^{\perp}=\{a\}=\overline{E_1}$.  
It follows that $(\overline{E_1}, \overline{E_2})$ is the unique decomposition of the block partition $(E_1,E_2)$. 
Similarly, for the blocks $\{a,d\}$, $\{b,d\}$ there are unique decompositions satisfying the condition from 
Definition \ref{normal-20}. Hence, $(X^{\prime},\perp)$ is a normal orthogonality space. 
\end{example}

$\vspace*{1mm}$

\section{Concluding Remarks} 

There exist implicative formulations of several logical algebras -- such as implicative-Boolean algebras 
for classical logic, implicative-ortholattices for orthologic, and \\ 
implicative-orthomodular lattices 
for quantum logic -- that provide implication-based presentations of these systems.  
Therefore, redefining the orthogonality relation on the basis of an implication constitutes a natural direction of 
research on the foundation of quantum mechanics. 
Within the framework of implicative-ortholattices, this paper defines and studies the implicative versions 
of orthogonality spaces, Sasaki projections, Sasaki set of projections, Dacey spaces and Sasaki spaces. 
Based on the Sasaki projections, we define the commutativity between two elements of an implicative-ortholattice,  
which plays an important role for the study of distributivity in these structures. 
In relation to this topic, we investigate the implicative-Boolean center as the ``classical core" of an implicative-orthomodular lattice, and we characterize the implicative-orthomodular lattices in terms of their  implicative-Boolean subalgebras. 
We also introduce the Sasaki sets of projections on implicative-ortholattices, proving that an implicative-ortholattice is an implicative-orthomodular lattice if and only if it admits a full Sasaki set of projections.
Finally, all these results are used to redefine the Dacey, Sasaki, and normal orthogonality spaces and to present 
some of their properties within the framework of implicative-ortholattices. \\
We presented these notions and results with the intention of showing that, within the framework of orthogonality over implicative-ortholattices, some properties of quantum logic can be proved more simply and more naturally, 
and that new properties could be explored. 
In orthomodular lattices, orthogonality is defined purely as a relation derived from lattice order and orthocomplementation, whereas in implicative-orthomodular lattices it is characterized by the implication operation,  offering a richer logical perspective. \\
In addition to introducing a new framework for studying orthogonal spaces, we highlight several specific contributions 
of this paper to the study of quantum structures: \\
$-$ new characterizations of implicative-orthomodular lattices (Theorem \ref{dioml-05-60}, Proposition \ref{ooml-40},  
Theorem \ref{sioml-40-10}, Corollary \ref{cioml-90-10}); \\
$-$ characterizations of commutativity between two elements of an implicative-orthomodular lattices (Corollary \ref{cioml-60-20}, Theorem \ref{cioml-60-40}); \\
$-$ characterizations of implicative-Boolean algebras (Theorems \ref{cioml-100}, \ref{cioml-100-10}, 
\ref{cioml-100-20}, Corollary \ref{cioml-100-30}). \\
Following the results of this paper, we suggest potential topics for future research: \\
$(1)$ Based on Sasaki projections, one can define and study adjointness and residuation in implicative-ortholattices 
as powerful tools for expressing orthogonality algebraically, and for relating quantum logic to other algebraic structures. In a future research paper, we will adress this topic. \\ 
$(2)$ Probability theory can serve as an important bridge between algebraic structures (logics) and physical reality (measurement outcomes), and it can be developed and investigated within the framework of implicative-ortholattices.

$\vspace*{1mm}$

\begin{center}
\sc Acknowledgement 
\end{center}
The author is very grateful to the anonymous referees for their useful remarks and suggestions on the subject that helped improving the presentation.

$\vspace*{1mm}$

\end{document}